\documentclass[12pt,a4paper,oneside]{article}

\usepackage[utf8]{inputenc}
\usepackage[russian,english]{babel}
\usepackage{amsthm,amsmath,amssymb}
\usepackage{xcolor}
\usepackage[colorlinks=true, linkcolor=blue]{hyperref}
\usepackage[english]{cleveref}
\usepackage{caption,subcaption}

\usepackage{mathtools}
\newcommand{\defeq}{\vcentcolon=}

\newcommand{\DmdOp}{\Diamond}
\newcommand{\BoxOp}{\Box}
\usepackage{thmtools}
\usepackage{thm-restate}

\newcommand{\mkTheorem}[4]{
  \declaretheorem[#4,name=#2,refname={#2,#3},Refname={#2,#3}]{#1}
}

\mkTheorem{definition}{Definition}{Definitions}{style=definition,numbered=no}
\mkTheorem{theorem}{Theorem}{Theorems}{numberwithin=section}
\mkTheorem{proposition}{Proposition}{Propositions}{sibling=theorem}
\mkTheorem{lemma}{Lemma}{Lemmas}{sibling=theorem}
\mkTheorem{corollary}{Corollary}{Corollaries}{sibling=theorem}
\crefname{equation}{}{}
\Crefname{equation}{}{}
\crefname{enumi}{}{}
\Crefname{enumi}{}{}
\crefname{section}{\S}{\S\S}

\hypersetup{hypertexnames=false}

\usepackage{multicol}
\usepackage{wrapfig}
\usepackage{graphicx}
\usepackage{centernot}

\usepackage{cases}

\newcommand{\E}[1]{\exists #1\,}
\newcommand{\A}[1]{\forall #1\,}

\usepackage{stmaryrd}

\usepackage{enumitem}
\setlist[enumerate, 1]{label=(\roman*)}

\makeatletter
\let\oldvec\vec
\renewcommand\vec[1]{%
    \renewcommand\sp{\ifx t#1\mskip1.4\thinmuskip\else\mskip0.7\thinmuskip\fi}%
    \oldvec{#1}%
    \@ifnextchar){\sp}{%
        \@ifnextchar'{\sp}{%
            \@ifnextchar]{\sp}{}%
        }%
    }%
}
\makeatother

\author{Lev Dvorkin\thanks{The author is a winner of the ``Junior Leader''
competition grant held by the ``BASIS'' Foundation for the Development of
Theoretical Physics and Mathematics.}}
\title{Monotonicity versus positivity in modal logics}
\date{\today}

\begin{document}
  \maketitle \begin{abstract}
  We say that a logic $\Lambda$ has the Lyndon positivity property (LPP) if all formulas which are
  monotone in $\Lambda$ (that is, are preserved under increasing the valuation on
  $\Lambda$-algebras) are $\Lambda$-equivalent to positive formulas (formulas without negation and
  implication symbols). In the present paper, we investigate LPP in propositional monotone modal
  logics. First, we transfer Lyndon's result from classical predicate calculus and prove LPP for all
  normal modal logics with the Lyndon interpolation property (LIP). Then we prove that all logics
  between $\mathrm{K4}.3$ and $\mathrm{S4}.3$ do not have LPP. We also show that among tabular
  extensions of $\mathrm{S4}$ there are infinitely many logics with LPP and infinitely many logics
  without this property. Finally, we prove that all canonical monotone modal logics which are
  preserved under bisimulation products have both LIP and LPP. In particular, we show LIP and LPP
  for all logics that are axiomatizable over the minimal monotone logic $\mathrm{EM}$ by means of
  closed formulas and formulas of the form $\alpha(p) \to \DmdOp p$, where $\alpha$ is positive.
\end{abstract}
\section{Introduction}
  \label{s:intro} It is well-known that disjunction, conjunction, and constant functions form a
  complete system in the class of monotone boolean functions. To state this fact in logical terms,
  one can use the following definitions: a boolean formula $\varphi$ is called
  \begin{itemize}
    \item \emph{monotone} if it is preserved when one increase valuation of variables;
    \item \emph{positive} if it is built from variables and constants using only $\wedge$ and
          $\vee$.
  \end{itemize}
  All positive formulas are trivially monotone. The above fact modulo completeness theorem yields
  the converse: in the classical propositional logic, every monotone formula is equivalent to a
  positive one.

  Lyndon~\cite{Lynd59Hom} extended this result to classical predicate calculus (CPC). A first-order
  formula is
  \begin{itemize}
    \item \emph{monotone} if it is preserved under increasing valuation of predicates;
    \item \emph{positive} if it is built from variables and constants using only $\wedge$, $\vee$,
          and quantifiers.
  \end{itemize}
  It follows from~\cite[Proposition 2]{Lynd59Hom} that all monotone formulas in CPC are equivalent
  to positive ones. Observe that this proposition is itself a straightforward consequence of the
  Lyndon interpolation property (LIP), which was established for CPC in~\cite{Lynd59Int}.

  In general, given a logic $\Lambda$, one may ask whether it satisfies the following \emph{Lyndon
  positivity property (LPP)}: every formula which is monotone in $\Lambda$ is $\Lambda$-equivalent
  to a positive formula. Of course, to consider LPP for a given logic, one needs to define what
  monotone and positive formulas are in its language, but in most cases this can be done easily by
  analogy with the above definitions for classical propositional and predicate logics.

  Ajtai and Gurevich~\cite{AjtGur87} showed that LPP does not hold in classical first-order finite
  model theory: there is a first-order formula $\varphi(P)$ that is preserved under increasing
  valuations of $P$ in all finite models, but is not equivalent over the class of all finite models
  to any positive (with respect to $P$) formula. Note that it is a folklore fact that LIP, and even
  weaker Craig interpolation property (CIP), do not hold in finite model theory.

  In~\cite{DAgHol00}, it was shown that LPP holds in the modal $\mu$-calculus, which is the
  extension of the minimal normal modal logic $\mathrm{K}$ by the least and the greatest fixed-point
  operators. Notice that $\mu$-calculus also has LIP~\cite{AfshLeigh22}.

  In the present paper, we explore the Lyndon positivity property in modal logics. In fact, LPP for
  the (polymodal) logic $\mathrm{K}$ was established in~\cite[Section 6.6]{dRij93}. However, no
  results on LPP in other modal systems are known to the author. For a modal logic $\Lambda$, we say
  that a modal formula $\varphi$ is
  \begin{itemize}
    \item \emph{monotone in $\Lambda$} if it is preserved under increasing of valuation in
          $\Lambda$-algebras;
    \item \emph{positive} if it is built from variables and constants using only $\vee$, $\wedge$,
          $\DmdOp$, and $\BoxOp$.
  \end{itemize}
  If the logic is monotone (that is, it contains the axiom $\DmdOp p \to \DmdOp(p \vee q)$), then
  all positive formulas are monotone. Moreover, following the Lyndon's proof one can easily show
  that all normal modal logics with LIP have LPP (\cref{lynd}). This result, although very simple,
  is the cornerstone of the present paper: all other results may be viewed as refinements of
  specific aspects of this theorem. We show that
  \begin{itemize}
    \item LIP is an essential condition: there are normal logics without LIP and LPP. In particular,
          the logic of the two-element cluster, which is known to have CIP, but not LIP, does not
          have LPP (\cref{ex:cn-par}). Even more natural examples are $\mathrm{K4}.3$ and
          $\mathrm{S4}.3$ which are known to lack CIP. We will show that they (and all logics
          between them) lack LPP too (\cref{ex:lin}).
    \item LIP is not a necessary condition: there are normal logics without LIP, in which LPP still
          holds. Moreover, we will show that there are infinitely many such logics among tabular
          extensions of $\mathrm{S4}$ (\cref{ex:dk}). At the same time, it is well-known that only
          finitely many extensions of $\mathrm{S4}$ have CIP~\cite{Maks79}.
    \item Normality is not a necessary condition. More specifically, in \cref{s:non-norm}, we will
          show that all $\sigma$- and $\pi$-canonical monotone modal which are preserved under
          bisimulation products have both LIP and LPP (it is known that all such logics have CIP,
          see~\cite[Section 5.2]{Marx95} for normal logics and~\cite[Section 9.2]{Han03} for
          arbitrary monotone logics).

          In particular, we will show that all logics axiomatizable over the minimal monotone logic
          $\mathrm{EM}$ by means of closed formulas and formulas of the form
          $\alpha(p) \to \DmdOp p$, where $\alpha$ is positive, satisfy these conditions and, hence,
          have both LIP and LPP.
  \end{itemize}
  One can see that there is one option left here: it is unclear whether the normality condition is
  essential, that is, whether LIP implies LPP for non-normal monotone logics. Counterexamples would
  be of great interest here, since LIP implies LPP in all systems considered so far.
 \section{Preliminaries}
  \subsection{Modal formulas}
    \label{ss:formulas} Let us fix some countable set $\mathcal{V}$ of variables. We use letters
    $p, q, r$ with indexes as meta-variables over $\mathcal{V}$ and assume, if the converse is not
    stated explicitly, that different letters and letters with different indexes correspond to
    different variables. A \emph{literal} is a variable $p \in \mathcal{V}$ or its negation
    $\neg p$. Negation of a literal is defined in an obvious way ($\neg(\neg p) = p$). For a set of
    literals $\tau$, we denote
    \begin{gather*}
      \neg\tau \defeq \{\neg l \mid l \in \tau\},\quad \tau^\pm \defeq \tau \cup \neg\tau.
    \end{gather*}
    Notice that $\mathcal{V}^\pm$ is the set of all literals. Sometimes, we use a tuple of variables
    $\vec p = (p_i)_{i < n}$ for the set $\{p_0, \dots, p_{n-1}\}$ and write something like
    $q \in \neg\vec p \cup \vec r ^{\pm}$.

    \emph{Modal formulas} are built from literals $l \in \mathcal{V}^\pm$ and constants $\bot$,
    $\top$ using binary connectives $\wedge$, $\vee$ and unary connectives $\DmdOp$, $\BoxOp$. The
    set of all formulas is denoted by $\mathrm{Fm}$. Negation of a formula is defined using duality
    laws. For a set of formulas $\Gamma$, we denote by $\neg\Gamma$ the set
    $\{\neg\varphi \mid \varphi \in \Gamma\}$. A formula is \emph{positive} if it does not contain
    negation of variables. For a formula $\varphi$, we denote by $\mathrm{vars}(\varphi)$ and
    $\mathrm{lits}(\varphi)$ the sets of all variables and literals from $\varphi$ respectively.
    Notice that $p \in \mathrm{vars}(\neg p) \setminus \mathrm{lits}(\neg p)$ and
    $\mathrm{vars}(\varphi)^\pm = \mathrm{lits}(\varphi)^\pm$ for every $\varphi \in \mathrm{Fm}$.
    For $\tau \subseteq \mathcal{V}^\pm$, we denote by $\mathrm{Fm}_\tau$ the set of all formulas
    $\varphi$ such that $\mathrm{lits}(\varphi) \subseteq \tau$. Notice that
    $\mathrm{Fm}_\mathcal{V}$ is the set of all positive formulas and
    $\mathrm{Fm}_{\neg\tau} = \neg\mathrm{Fm}_\tau$ for all $\tau \subseteq \mathcal{V}^\pm$. When
    we consider a formula $\varphi(\vec p)$, where $\vec p = (p_i)_{i<n}$, we mean some formula
    $\varphi$ such that $\mathrm{vars}(\varphi) \subseteq \vec p$. In this case, for a tuple of
    formulas $\vec \eta = (\eta_i)_{i<n}$, we denote by $\varphi(\vec \eta)$ the formula which is
    obtained by simultaneous substitution of $\eta_i$ for $p_i$ (and expanding all negations of
    formulas). The following lemma is clear:
    \begin{lemma}
      \label{l:nnf} Let $\varphi$ be a formula, $\vec p = (p_i)_{i<n}$ and $\bar q = (q_j)_{j<m}$ be
      tuples of variables such that $\mathrm{lits}(\varphi) = \vec p \cup \neg\bar q$ (some $p_i$
      and $q_j$ can coincide). Then there is a (unique) positive formula
      $\alpha(r_0, \dots, r_{n+m-1})$ such that $\varphi = \alpha(\vec p, \neg\bar q)$.
    \end{lemma}
  \subsection{Modal logics}
    A \emph{modal logic} is a set $\Lambda \subseteq \mathrm{Fm}$ which contains all classical
    tautologies and is closed under the rules of modus ponens, substitution, and \emph{equivalent
    replacement} $\frac{\varphi \leftrightarrow \psi}{\DmdOp\varphi \leftrightarrow \DmdOp\psi}$.
    The least logic is denoted by $\mathrm{E}$. The greatest logic is the set of all formulas
    $\mathrm{Fm}$. All logics $\Lambda \neq \mathrm{Fm}$ are called \emph{consistent}. A formula
    $\varphi$ \emph{is derivable in $\Lambda$ from $\Gamma \subseteq \mathrm{Fm}$} if there is a
    finite subset $\Gamma_0 \subseteq \Gamma$ such that $\bigwedge\Gamma_0 \to \varphi$ is in
    $\Lambda$. In this case, we write $\Gamma \vdash_\Lambda \varphi$. If
    $\emptyset \vdash_\Lambda \varphi$ we say that $\varphi$ is \emph{derivable} in $\Lambda$ and
    write $\Lambda \vdash \varphi$. Clearly,
    $\Lambda \vdash \varphi \Leftrightarrow \varphi \in \Lambda$. For a logic $\Lambda$ and a set
    $\Gamma \subseteq \mathrm{Fm}$, we denote by $\Lambda + \Gamma$ the smallest logic containing
    $\Lambda \cup \Gamma$. We use the following notation for some modal formulas:
    \begin{gather*}
      {\rm AM} \defeq \DmdOp p \to \DmdOp(p \vee q),\quad {\rm AC} \defeq \DmdOp(p \vee q) \to
      \DmdOp p \vee \DmdOp q\\
      {\rm AN} \defeq \BoxOp\top,\quad {\rm AP} \defeq \DmdOp\top,\quad {\rm AD} \defeq \BoxOp p \to
      \DmdOp p,\\
      {\rm AT} \defeq p \to \DmdOp p,\quad {\rm A}4 \defeq \DmdOp\DmdOp p \to \DmdOp p,\quad
      {\rm AB} \defeq \DmdOp p \to \BoxOp\DmdOp p,\\
      {\rm A}.3 \defeq \DmdOp p \wedge \DmdOp q \to \DmdOp(p \wedge \DmdOp q) \vee \DmdOp(q \wedge
      \DmdOp p) \vee \DmdOp(p \wedge q).
    \end{gather*}
    For a logic $\Lambda$ and a formula ${\rm A}X$ from this list, we denote by $\Lambda X$ the
    logic $\Lambda + \{{\rm A}X\}$. We also use the following standard notation:
    \begin{gather*}
      \mathrm{K} \defeq {\rm \mathrm{EM} NC},\quad \mathrm{S4} \defeq {\rm \mathrm{K} T4},\quad
      \mathrm{S5} \defeq {\rm \mathrm{K} TB4}.
    \end{gather*}
    Logics containing $\mathrm{EM}$ are called \emph{monotone}, logics containing $\mathrm{K}$ are
    called \emph{normal}.

    For a logic $\Lambda$, we consider the following relations on formulas:
    \begin{gather*}
      \varphi \preceq_\Lambda \psi :\Leftrightarrow \Lambda \vdash \varphi \to \psi ,\quad \varphi
      \sim_\Lambda \psi :\Leftrightarrow \Lambda \vdash \varphi \leftrightarrow \psi.
    \end{gather*}
    If $\varphi \sim_\Lambda \psi$, then we say that $\varphi$ and $\psi$ are \emph{equivalent in
    $\Lambda$} or, shorter, \emph{$\Lambda$-equivalent}.
  \subsection{Interpolation properties}
    Let $\Lambda$ be a modal logic, $\varphi$ and $\psi$ be formulas. We say that
    $\iota \in \mathrm{Fm}$ is an \emph{interpolant for $\varphi$ and $\psi$ in $\Lambda$} if
    $\varphi \preceq_\Lambda \iota \preceq_\Lambda \psi$. An interpolant $\iota$ is a \emph{Craig}
    interpolant if $\mathrm{vars}(\iota) \subseteq \mathrm{vars}(\varphi) \cap \mathrm{vars}(\psi)$.
    It is a \emph{Lyndon} interpolant if
    $\mathrm{lits}(\iota) \subseteq \mathrm{lits}(\varphi) \cap \mathrm{lits}(\psi)$. Clearly, if
    $\iota$ is an interpolant, then $\varphi \preceq_\Lambda \psi$. $\Lambda$ has \emph{Craig
    (Lyndon) interpolation property} if a Craig (Lyndon) interpolant for $\varphi$ and $\psi$ exists
    whenever $\varphi \preceq_\Lambda \psi$. We use standard abbreviations CIP and LIP for Craig and
    Lyndon interpolation properties respectively. Clearly, every Lyndon interpolant for $\varphi$
    and $\psi$ is a Craig interpolant for the same formulas. Therefore, every logic with LIP has
    CIP.
  \subsection{Algebraic semantics}
    A \emph{modal algebra} is a pair $\mathfrak{A} = (\mathfrak{A}_0, \DmdOp)$, where
    $\mathfrak{A}_0 = (A; \bot, \top, \vee, \wedge, \neg)$ is a boolean algebra and
    $\DmdOp : A \to A$ is an unary operator. A \emph{valuation} on $\mathfrak{A}$ is a function
    $\vartheta : \mathcal{V} \to A$. $\vartheta$ is extended uniquely to homomorphism from the term
    algebra $\mathfrak{Fm}$ to $\mathfrak{A}$ which will be denoted by the same symbol $\vartheta$.
    An \emph{algebraic model on $\mathfrak{A}$} is a pair
    $\mathfrak{M} = (\mathfrak{A}, \vartheta)$, where $\vartheta$ is a valuation on $\mathfrak{A}$.
    A formula $\varphi$ is \emph{true in $\mathfrak{M}$} if $\vartheta(\varphi) = \top$, is
    \emph{valid in $\mathfrak{A}$} if $\varphi$ is true in all models on $\mathfrak{A}$. We denote
    this by $\mathfrak{M} \vDash \varphi$ and $\mathfrak{A} \vDash \varphi$ respectively.

    For a logic $\Lambda$, we say that $\mathfrak{M}$ is a $\Lambda$-model and $\mathfrak{A}$ is a
    $\Lambda$-algebra if $\mathfrak{M} \vDash \Lambda$ and $\mathfrak{A} \vDash \Lambda$
    respectively. Clearly, every model on a $\Lambda$-algebra is a $\Lambda$-model. The \emph{logic
    of algebra $\mathfrak{A}$} is the set $\mathrm{Log}\, \mathfrak{A}$ of all formulas that are
    valid in $\mathfrak{A}$. It is easy to show that $\mathrm{Log}\, \mathfrak{A}$ is a logic for
    every algebra $\mathfrak{A}$. The converse also holds: for a logic $\Lambda$, relation
    $\sim_\Lambda$ is a congruence on $\mathfrak{Fm}$. The quotient structure
    $\mathfrak{A}_\Lambda \defeq \mathfrak{Fm}/{\sim}_\Lambda$ is a modal algebra which is called
    the \emph{Lindenbaum-Tarski algebra of $\Lambda$}. It is well-known that
    $\mathrm{Log}\, \mathfrak{A}_\Lambda = \Lambda$.

    For a boolean algebra $\mathfrak{A}_0$, relation $a \leq b :\Leftrightarrow a \vee b = b$ is a
    partial order. A modal algebra $(\mathfrak{A}_0, \DmdOp)$ is \emph{monotone} if $\DmdOp$
    preserves this partial order. It is easy to see that $\mathfrak{A}$ is monotone iff
    $\mathrm{Log}\, \mathfrak{A}$ is monotone.

    Let us fix some tuple of variables $\vec p = (p_i)_{i<n}$ and modal algebra $\mathfrak{A}$. We
    consider the algebra $\mathfrak{A}^{A^n}$ of all $n$-place operators $f : A^n \to A$ with
    pointwise operations. For $j < n$, let $(p_j)_\mathfrak{A} : A^n \to A$ be the \emph{projection
    on the $j$-th coordinate} $(a_i)_{i < n} \mapsto a_j$. We extend this mapping to the unique
    homomorphism ${\cdot}_{\mathfrak{A}} : \mathfrak{Fm}_{\vec p^\pm} \to \mathfrak{A}^{A^n}$, where
    $\mathfrak{Fm}_{\vec p^\pm}$ is a subalgebra of $\mathfrak{Fm}$ with the carrier
    $\mathrm{Fm}_{\vec p^\pm}$. For example, for $\vec p = (p_1, p_2)$,
    $(p_1 \vee \DmdOp p_2)_\mathfrak{A}$ maps $(a_1, a_2) \in A^2$ to $a_1 \vee \DmdOp a_2$. It is
    easy to see that formulas $\varphi(\vec p)$ and $\psi(\vec p)$ are $\Lambda$-equivalent iff
    $\varphi_\mathfrak{A} = \psi_\mathfrak{A}$.

    Notice that the operator $\varphi_\mathfrak{A}$ is well-defined only when we fix the
    tuple~$\vec p$ (including the order of variables in it). For example, $\varphi = p_1$ defines
    the first projection for $\vec p = (p_1, p_2)$ and the second projection for
    $\vec p = (p_2, p_1)$.
  \subsection{Relations, functions and operators}
    For a set $W$, we denote by $\mathcal{P}(W)$ the collection of all its subsets. The boolean
    algebra $(\mathcal{P}(W), \emptyset, W, \cup, \cap, -)$, where $-X \defeq W \setminus X$ for
    $X \subseteq W$, is denoted by $\mathfrak{P}(W)$.

    A \emph{relation between} sets $W_1$ and $W_2$ is a set $Z \subseteq W_1 \times W_2$. If
    $W_1 = W_2 = W$, we say that $Z$ is a \emph{relation on $W$}. We identify relation $Z$ with the
    operator $Z : \mathcal{P}(W_1) \to \mathcal{P}(W_2)$ mapping $X_1 \subseteq W_1$ to its full
    image $\{w_2 \in W_2 \mid \E{ w_1 \in X_1}(w_1 \mathrel Z w_2)\}$. The inverse relation
    $\{(v, u) \mid u \mathrel Z v\}$ is denoted by $Z^{-1}$ (and is identified with the full
    preimage operator). For relations $R_1 \subseteq W_0 \times W_1$ and
    $R_2 \subseteq W_1 \times W_2$, their \emph{composition} is the relation
    \begin{gather*}
      R_2R_1 = \{(w_0, w_2) \in W_0 \times W_2 \mid \E{ w_1 \in W_1}(w_0 \mathrel R_1 w_1
      \mathrel R_2 w_2)\}.
    \end{gather*}
    The \emph{domain of $Z$} is the set ${\rm dom}\, Z \defeq Z^{-1}W_2$, the \emph{range of $Z$} is
    the set ${\rm rng}\, Z \defeq ZW_1$. $Z$ is \emph{full} if ${\rm dom}\, Z = W_1$ and
    ${\rm rng}\, Z = W_2$. As usual, functions $f : W_1 \to W_2$ are treated as relations
    $\{(u, f(u)) \mid u \in W_1\}$. The identity function on $W$ is denoted by $1_W$ (and is
    identified with identity operator on $\mathcal{P}(W)$).

    It can be easily checked that, for every relation $Z \subseteq W_1 \times W_2$, $Z$ and
    $-Z^{-1}-$ form a monotone Galois connection between $\mathfrak{P}(W_1)$ and
    $\mathfrak{P}(W_2)$, that is, the following holds:
    \begin{gather*}
      \A{ X_1 \subseteq W_1}\A{ X_2 \subseteq W_2}(Z X_1 \subseteq X_2 \Leftrightarrow X_1 \subseteq
      -Z^{-1}{-}X_2),\\
      {-}Z^{-1}{-}Z \geq 1_{W_1},\quad Z{-}Z^{-1}{-} \leq 1_{W_2}.
    \end{gather*}
    For a function $f : W_1 \to W_2$, $f^{-1} : \mathfrak{P}(W_2) \to \mathfrak{P}(W_1)$ is a
    homomorphism of boolean algebras, whence $-f^{-1}- = f^{-1}$ and we have the following:
    $ff^{-1} \leq 1_{W_2}$, $f^{-1}f \geq 1_{W_1}$. If ${\rm rng}\, f = W_2$, then clearly
    $ff^{-1} = 1_{W_2}$.

  \subsection{Neighborhood and Kripke semantics}
    A \emph{neighborhood frame} is a pair $\mathfrak{F} = (W, \DmdOp)$, where
    $\DmdOp : \mathcal{P}(W) \to \mathcal{P}(W)$. Elements of $W$ are called \emph{worlds}. For a
    neighborhood frame $\mathfrak{F}$, we consider a modal algebra
    $\mathfrak{F}^* \defeq (\mathfrak{P}(W), \DmdOp)$ and transfer all semantic notions from
    $\mathfrak{F}^*$ to $\mathfrak{F}$. In particular, a \emph{neighborhood model} is a pair
    $\mathfrak{M} = (\mathfrak{F}, \vartheta)$, where $\vartheta$ is a valuation on
    $\mathfrak{P}(W)$. We say that a formula $\varphi$ \emph{is true at the world $w$ in the model
    $\mathfrak{M}$} and write $\mathfrak{M}, w \vDash \varphi$ if $w \in \vartheta(\varphi)$.

    A \emph{Kripke frame} is a pair $\mathcal{F} = (W, R)$, where $R$ is a relation on $W$, which is
    called the \emph{accessability relation} of $\mathcal{F}$. For a Kripke frame $\mathcal{F}$, we
    consider the neighborhood frame $\mathfrak{n}\mathcal{F} \defeq (W, R^{-1})$, and transfer all
    semantic notions from $\mathfrak{n}\mathcal{F}$ to $\mathcal{F}$. We also put
    $\mathcal{F}^* \defeq (\mathfrak{n}\mathcal{F})^*$.
  \subsection{Bisimulations}
    \label{ss:bisim}
    \begin{definition}
      Let $\mathfrak{F}_1 = (W_1, \DmdOp_1)$ and $\mathfrak{F}_2 = (W_2, \DmdOp_2)$ be monotone
      neighborhood frames. Relation $Z \subseteq W_1 \times W_2$ is a \emph{bisimulation between
      $\mathfrak{F}_1$ and $\mathfrak{F}_2$} if the following condition hold:
      \begin{itemize}
        \item $({\rm zig})$: for all $X_1 \subseteq W_1$, $Z(\DmdOp_1X_1) \subseteq \DmdOp_2Z(X_1)$;
        \item $({\rm zag})$: for all $X_2 \subseteq W_2$,
              $Z^{-1}(\DmdOp_2X_2) \subseteq \DmdOp_1Z^{-1}(X_2)$.
      \end{itemize}
    \end{definition}
    \begin{lemma}
      $({\rm zag})$ is equivalent to the following condition:
      \begin{itemize}
        \item $({\rm zag}')$: for all $X_1 \subseteq W_1$,
              $Z(\BoxOp_1X_1) \subseteq \BoxOp_2Z(X_1)$.
      \end{itemize}
    \end{lemma}
    \begin{proof}
      Suppose that $({\rm zag})$ is satisfied. Then
      $\DmdOp_2X_2 \subseteq {-}Z{-} \DmdOp_1Z^{-1} X_2$, whence
      \begin{gather*}
        Z\BoxOp_1{-}Z^{-1} X_2 \subseteq \BoxOp_2{-} X_2.
      \end{gather*}
      Let $X_2 \defeq {-}Z X_1$. Since ${-}Z^{-1}{-}Z \geq 1_{W_1}$,
      \begin{gather*}
        Z\BoxOp_1 X_1 \subseteq Z\BoxOp_1 {-}Z^{-1}{-}Z X_1 \subseteq \BoxOp_2Z X_1,
      \end{gather*}
      that is, $({\rm zag}')$ holds. To prove the converse implication, it is sufficient to consider
      relation $Z^{-1}$ between $(W_2, \BoxOp_2)$ and $(W_1, \BoxOp_1)$:
      $({\rm zag}) \Rightarrow ({\rm zag}')$ for this relation is exactly
      $({\rm zag}') \Rightarrow ({\rm zag})$ for $Z$.
    \end{proof}
    \begin{definition}
      Let $\mathfrak{M}_1 = (\mathfrak{F}_1, \vartheta_1)$ and
      $\mathfrak{M}_2 = (\mathfrak{F}_2, \vartheta_2)$ be models on monotone neighborhood frames
      $\mathfrak{F}_1 = (W_1, \DmdOp_1)$ and $\mathfrak{F}_2 = (W_2, \DmdOp_2)$, $\tau$ be a set of
      literals. Bisimulation $Z$ between $\mathfrak{F}_1$ and $\mathfrak{F}_2$ is a
      \emph{$\tau$-bisimulation between $\mathfrak{M}_1$ and $\mathfrak{M}_2$} if the following
      condition holds
      \begin{itemize}
        \item $({\rm lit})$: for all $l \in \tau$, $Z(\vartheta_1(l)) \subseteq \vartheta_2(l)$.
      \end{itemize}
    \end{definition}
    Clearly, if $Z$ is a $\tau$-bisimulation between $\mathfrak{M}_1$ and $\mathfrak{M}_2$, then
    $Z^{-1}$ is a $\neg\tau$-bisimulation between $\mathfrak{M}_2$ and $\mathfrak{M}_1$.

    We will use the following synonyms:
    \begin{itemize}
      \item \emph{bisimulations} (between models) for $\mathcal{V}^\pm$-bisimulations;
      \item \emph{directed bisimulations} for $\mathcal{V}$-bisimulations;
      \item \emph{$\vec p$-directed bisimulations} for $\tau$-bisimulations, where $\vec p$ is a
            tuple of variables and $\tau = \mathcal{V}^\pm \setminus \neg\vec p$.
    \end{itemize}
    The notion of bisimulation between monotone neighborhood models is standard (see,
    e.g.,~\cite[Definition 2.2]{Pac17}). For Kripke models, directed bisimulations were considered
    in the context of positive fragments~\cite{KurtDRij97}, $\sigma^\pm$-bisimulations for
    $\sigma \subseteq \mathcal{V}$ were considered in connection with CIP~\cite{KWZ23}.
    In~\cite{Kur20}, a notion which is equivalent to $\tau$-bisimulations for Kripke frames was used
    in the investigation of uniform LIP.

    Let $\mathfrak{M}_k = (W_k, \DmdOp_k, \vartheta_k)$, $k = 1, 2$ be neighborhood models,
    $\varphi$ be a modal formula. We say that relation $Z \subseteq W_1 \times W_2$ \emph{preserves
    $\varphi$} if $Z \vartheta_1(\varphi) \subseteq \vartheta_2(\varphi)$. In other words, $Z$
    preserves $\varphi$ if
    \begin{gather*}
      \A{(w_1, w_2) \in Z} (\mathfrak{M}_1, w_1 \vDash \varphi \Rightarrow \mathfrak{M}_2, w_2
      \vDash \varphi).
    \end{gather*}
    The following lemma can be easily proved by induction:
    \begin{lemma}
      \label{bisim-main} Let $Z$ be a $\tau$-bisimulation between neighborhood models. Then $Z$
      preserves all formulas from $\mathrm{Fm}_\tau$.
    \end{lemma}
    As always, we transfer the above notions from neighborhood to Kripke frames and models. For
    Kripke frames, $({\rm zig})$ and $({\rm zag})$ are equivalent to the following well-known
    conditions:
    \begin{itemize}
      \item $({\rm zig}_\mathrm{K})$ $\A{(w_1, w_2) \in Z}\A{ v_1 \in R_1\{w_1\}}\E{ v_2 \in
            R_2\{w_2\}}(v_1 \mathrel Z v_2)$;
      \item $({\rm zag}_\mathrm{K})$ $\A{(w_1, w_2) \in Z}\A{ v_2 \in R_2\{w_2\}}\E{ v_1 \in
            R_1\{w_1\}}(v_1 \mathrel Z v_2)$.
    \end{itemize}
  \subsection{Morphisms}
    \begin{definition}
      Let $\mathfrak{F}_1 = (W_1, \DmdOp_1)$ and $\mathfrak{F}_2 = (W_2, \DmdOp_2)$ be neighborhood
      frames. A mapping $f : W_1 \to W_2$ is a \emph{morphism} from $\mathfrak{F}_1$ to
      $\mathfrak{F}_2$ if
      \begin{gather}
        \A{ X_2 \subseteq W_2}\bigl(f^{-1}(\DmdOp_2X_2) = \DmdOp_1f^{-1}(X_2)\bigr).
        \label{eq:p-morph}
      \end{gather}
      In this case, we write $f : \mathfrak{F}_1 \to \mathfrak{F}_2$.
    \end{definition}
    In other words, $f$ is a morphism of Kripke frames iff
    $f^{-1} : \mathfrak{F}_2^* \to \mathfrak{F}_1^*$ is a homomorphism of modal algebras. Usually,
    such mappings are called p-morphisms, but this term could be confused with the above
    $p$-something terms in which $p$ is a variable.
    \begin{lemma}
      A function $f : W_1 \to W_2$ is a morphism from $\mathfrak{F}_1$ to $\mathfrak{F}_2$ iff $f$
      is a bisimulation between $\mathfrak{F}_1$ and $\mathfrak{F}_2$. In other words, morphisms of
      frames are exactly total functional bisimulations.
    \end{lemma}
    \begin{proof}
      Recall that, for a function $f : W_1 \to W_2$, $f^{-1}f \geq 1_{W_1}$ and
      $ff^{-1} \leq 1_{W_2}$. If $f$ is a morphism, then \cref{eq:p-morph} implies $({\rm zag})$ and
      \begin{gather*}
        f(\DmdOp_1X_1) \subseteq f\DmdOp_1f^{-1}fX_1 = ff^{-1}\DmdOp_2fX_1 \subseteq \DmdOp_2f(X_1).
      \end{gather*}
      If $f$ is a bisimulation, then, by $({\rm zag})$,
      $f^{-1}(\DmdOp_2X_2) \subseteq \DmdOp_1f^{-1}(X_2)$. Conversely,
      \begin{gather*}
        \DmdOp_1f^{-1}(X_2) \subseteq f^{-1}f \DmdOp_1 f^{-1} X_2 \subseteq f^{-1}\DmdOp_2
        ff^{-1}X_2 \subseteq f^{-1}(\DmdOp_2X_2).
      \end{gather*}
    \end{proof}
    \begin{definition}
      Let $\mathfrak{M}_1 = (\mathfrak{F}_1, \vartheta_1)$ and
      $\mathfrak{M}_2 = (\mathfrak{F}_2, \DmdOp_2)$ be neighborhood models, $\vec p$ be a tuple of
      variables. A morphism $f : \mathfrak{F}_1 \to \mathfrak{F}_2$ is called a
      \emph{($\vec p$-directed) morphism} from $\mathfrak{M}_1$ to $\mathfrak{M}_2$ if it is a
      ($\vec p$-directed) bisimulation between models $\mathfrak{M}_1$ and $\mathfrak{M}_2$.
    \end{definition}
    Notice that, if $f : \mathfrak{F}_1 \to \mathfrak{F}_2$ is a morphism and $\vartheta_2$ is a
    valuation on $\mathfrak{F}_2$, then $\vartheta_1 \defeq f^{-1}\vartheta_2$ is the unique
    valuation on $\mathfrak{F}_1$ such that $f$ is a morphism from $(\mathfrak{F}_1, \vartheta_1)$
    to $(\mathfrak{F}_2, \vartheta_2)$.

    As always, we transfer the above notions from neighborhood to Kripke frames and models. It is
    easy to check that, for Kripke frames $\mathcal{F}_k = (W_k, R_k)$, \cref{eq:p-morph}~is
    equivalent to the following, well-known, conditions:
    \begin{itemize}
      \item $({\rm fwd})$:
            $\A{ w_1 \in W_1}\A{ v_1 \in R_1\{w_1\}} \bigl(f(w_1) \mathrel R_2 f(v_1)\bigr)$;
      \item $({\rm bwd})$: $\A{ w_1 \in W_1}\A{ v_2 \in R_2\{f(w_1)\}}\E{ v_1 \in
            R\{w_1\}}\bigl(f(v_1) = v_2\bigr)$.
    \end{itemize}
  \subsection{Canonical frames and models}
    Let us fix some logic $\Lambda$. A set of formulas $\Gamma$ is \emph{$\Lambda$-consistent} if
    $\Gamma \nvdash_\Lambda \bot$. We denote by $W_\Lambda$ the collection of all maximal (with
    respect to $\subseteq$) $\Lambda$-consistent sets. We put
    $\llbracket\Gamma\rrbracket_\Lambda \defeq \{w \in W_\Lambda \mid \Gamma \subseteq w\}$ and
    $\llbracket\varphi\rrbracket_\Lambda \defeq \llbracket\{\varphi\}\rrbracket_\Lambda$. Usually,
    we will omit the subscript and write simply $\llbracket\Gamma\rrbracket$ and
    $\llbracket\varphi\rrbracket$. The following results are well-known:
    \begin{itemize}
      \item (Lindenbaum's lemma): if $\Gamma$ is $\Lambda$-consistent, then
            $\llbracket\Gamma\rrbracket \neq \emptyset$;
      \item (representation of the Lindenbaum-Tarski algebra): let $D_\Lambda$ be the set
            $\{\llbracket\varphi\rrbracket \mid \varphi \in \mathrm{Fm}\}$, $\DmdOp_\Lambda$ be an
            operator on $D_\Lambda$ which maps $\llbracket\varphi\rrbracket$ to
            $\llbracket\DmdOp\varphi\rrbracket$ for all $\varphi \in \mathrm{Fm}$. Then the mapping
            $\llbracket\cdot\rrbracket_\Lambda : \mathrm{Fm} \to \mathcal{P}(W_\Lambda)$ induces the
            isomorphism between $\mathfrak{A}_\Lambda$ and $\mathfrak{D}_\Lambda$, where
            $\mathfrak{D}_\Lambda = (\mathfrak{P}(W_\Lambda)|_{D_\Lambda}, \DmdOp_\Lambda)$.
    \end{itemize}
    Notice that the dual operator $\BoxOp_\Lambda = -\DmdOp_\Lambda\!-$ maps
    $\llbracket\varphi\rrbracket$ to $\llbracket\BoxOp\varphi\rrbracket$.
    \begin{definition}
      A neighborhood frame $\mathfrak{F} = (W_\Lambda, \DmdOp)$ is a \emph{canonical frame for
      $\Lambda$} if $\DmdOp$ is an extension of $\DmdOp_\Lambda$, that is,
      $\DmdOp\llbracket\varphi\rrbracket = \llbracket\DmdOp\varphi\rrbracket$ for all
      $\varphi \in \mathrm{Fm}$. A canonical frame $\mathfrak{F}$ with valuation
      $\vartheta_\Lambda : p \mapsto \llbracket p\rrbracket$ is called a \emph{canonical model of
      $\Lambda$}.
    \end{definition}
    The following standard lemma can be easily checked by induction:
    \begin{lemma}
      \label{canon-main} Let $\mathfrak{M} = (W_\Lambda, \DmdOp, \vartheta_\Lambda)$ be a canonical
      neighborhood model for $\Lambda$. Then
      $\vartheta_\Lambda(\varphi) = \llbracket\varphi\rrbracket$ for all $\varphi \in \mathrm{Fm}$.
      In particular, $\mathfrak{M} \vDash \varphi \Leftrightarrow \varphi \in \Lambda$.
    \end{lemma}
    If $\Lambda$ is monotone, than $\DmdOp_\Lambda$ is a monotone operator. In general, a monotone
    operator $f$ on $\mathfrak{D}_\Lambda$ can be extended to a monotone operator on
    $\mathfrak{P}(W_\Lambda)$ in many ways. The following extensions trace back to the classical
    work of J\'onsson and Tarski~\cite{JonTar51} and are known as $\sigma$- and $\pi$-extensions:
    \begin{gather*}
      f^\sigma X \defeq \bigcup_{\Gamma : \llbracket\Gamma\rrbracket \subseteq X}\, \bigcap_{\varphi
      : \llbracket\Gamma\rrbracket \subseteq \llbracket\varphi\rrbracket}
      f\llbracket\varphi\rrbracket,\\
      f^\pi X \defeq \bigcap_{\Gamma : X \subseteq -\llbracket\Gamma\rrbracket}\, \bigcup_{\varphi :
      \llbracket\varphi\rrbracket \subseteq -\llbracket\Gamma\rrbracket}
      f\llbracket\varphi\rrbracket.
    \end{gather*}
    One can easily check that (see \cite[Section 6.1]{Han03})
    \begin{enumerate}
      \item $f^\sigma$ and $f^\pi$ are monotone operators. \label{ce:mon}
      \item $f^\sigma\llbracket\varphi\rrbracket = f^\pi\llbracket\varphi\rrbracket =
            f\llbracket\varphi\rrbracket$ for all $\varphi \in \mathrm{Fm}$. \label{ce:ext}
      \item $f^\sigma\llbracket\Gamma\rrbracket = \bigcap_{\varphi : \llbracket\Gamma\rrbracket
            \subseteq \llbracket\varphi\rrbracket} f\llbracket\varphi\rrbracket$ and
            $f^\pi(-\llbracket\Gamma\rrbracket) = \bigcup_{\varphi : \llbracket\varphi\rrbracket
            \subseteq -\llbracket\Gamma\rrbracket} f\llbracket\varphi\rrbracket$. \label{ce:cl}
      \item $f^\sigma X = \bigcup_{\Gamma : \llbracket\Gamma\rrbracket \subseteq X} f^\sigma
            \llbracket\Gamma\rrbracket$ and $f^\pi X = \bigcap_{\Gamma : X \subseteq
            -\llbracket\Gamma\rrbracket} f^\pi(-\llbracket\Gamma\rrbracket)$. \label{ce:arb}
      \item the dual of the $\sigma$- ($\pi$-)extension of the operator $f$ is \label{ce:dual} the
            $\pi$- ($\sigma$-)extension of the dual of $f$. In particular,
            $-\DmdOp^\sigma_\Lambda- = \BoxOp^\pi_\Lambda$ and
            $-\DmdOp^\pi_\Lambda- = \BoxOp^\sigma_\Lambda$.
    \end{enumerate}
    Notice that \cref{ce:cl,ce:arb} provide an equivalent two-step definition of $\sigma$- and
    $\pi$-extensions. First, we define $f^\sigma$ on the sets of the form
    $\llbracket\Gamma\rrbracket$ and $f^\pi$ on the sets of the form $-\llbracket\Gamma\rrbracket$.
    Then, we extend $f^\sigma$ and $f^\pi$ to arbitrary sets.
    \begin{definition}
      $\mathfrak{F}^\sigma_\Lambda \defeq (W_\Lambda, \DmdOp^\pi_\Lambda)$ and
      $\mathfrak{F}^\pi_\Lambda \defeq (W_\Lambda, \DmdOp^\sigma_\Lambda)$ are called the $\sigma$-
      and $\pi$-canonical models of $\Lambda$ respectively. A logic $\Lambda$ is $\sigma$-
      ($\pi$-)canonical if $\mathfrak{F}^\sigma_\Lambda \vDash \Lambda$
      ($\mathfrak{F}^\pi_\Lambda \vDash \Lambda$).
    \end{definition}
    Notice the discrepancy in the last definition: the operator on the $\sigma$-canonical frame is
    the $\pi$-canonical extension of $\DmdOp_\Lambda$ and vice versa. This is because usually
    neighborhood models are defined using $\BoxOp$ as a primary operator and the box-operator on
    $\sigma$-canonical frame is $\BoxOp^\sigma_\Lambda$ (see \cref{ce:dual}).
    \begin{definition}
      A formula $\varphi$ is \emph{$\sigma$- ($\pi$-)canonical} if
      $\mathfrak{F}^\sigma_\Lambda \vDash \varphi$ ($\mathfrak{F}^\pi_\Lambda \vDash \varphi$) for
      all $\Lambda \supseteq \mathrm{EM} + \varphi$.
    \end{definition}
    By \cref{ce:mon}, $\mathrm{EM}$ is both $\sigma$- and $\pi$-canonical. Moreover, if all formulas
    $\varphi \in \Gamma$ are $\sigma$- ($\pi$-)canonical, then the logic $\mathrm{EM} + \Gamma$ is
    $\sigma$- ($\pi$-)canonical. Since the truth of closed formulas does not depend on valuation of
    variables, from \cref{canon-main}, we immediately get that all closed formulas are both
    $\sigma$- and $\pi$-canonical. The following result was obtained in~\cite[Theorems 10.34 and
    10.44]{Han03}:
    \begin{proposition}
      \label{KW-canon} Formulas of the form
      \begin{gather*}
        \chi \to \bigvee_{i<n}p_i \vee \bigvee_{j<m}\#q_j,
      \end{gather*}
      where $\chi$ is positive and $\# = \DmdOp$ ($\# = \BoxOp$), are $\sigma$- ($\pi$-)canonical.
    \end{proposition}
    In particular, ${\rm AT}$, ${\rm A}4$, ${\rm AC}$, and ${\rm AD}$ are $\sigma$-canonical,
    ${\rm AT}$ is also $\pi$-canonical since
    $\mathrm{EM} + {\rm AT} = \mathrm{EM} + \BoxOp p \to p$.

    If $\Lambda$ is normal, then $\DmdOp^\sigma_\Lambda = \DmdOp^\pi_\Lambda = R^{-1}_\Lambda$,
    where $R_\Lambda$ is a relation on $W_\Lambda$ defined by the following well-known condition:
    \begin{gather*}
      w \mathrel R_\Lambda v :\Leftrightarrow \A{\varphi \in v}(\DmdOp\varphi \in w).
    \end{gather*}
    $\mathcal{F}_\Lambda \defeq (W_\Lambda, R_\Lambda)$ and
    $\mathcal{M}_\Lambda \defeq (\mathcal{F}_\Lambda, \vartheta_\Lambda)$ are known as the
    \emph{canonical Kripke frame} and the \emph{canonical Kripke model of $\Lambda$} respectively.
    Notice that
    $\mathfrak{n}\mathcal{F}_\Lambda = \mathfrak{F}^\sigma_\Lambda = \mathfrak{F}^\pi_\Lambda$.
    Hence, for normal logics, we have only one canonical frame and only one notion of canonicity:
    $\Lambda$ is \emph{canonical} iff $\mathcal{F}_\Lambda \vDash \Lambda$. However, one should be
    careful with canonical formulas: $\varphi$ is canonical in the sense of Kripke semantics (that
    is, the logic $\mathrm{K} + \varphi$ is canonical) iff ${\rm AC} \wedge {\rm AN} \wedge \varphi$
    is $\sigma$- (or $\pi$-)canonical.
  \subsection{Modal duality in monotone logics}
    For a formula $\varphi$, consider its dual formula $\varphi^d$ which is obtained from $\varphi$
    by replacing $\DmdOp$ by $\BoxOp$ and vice versa. For a neighborhood frame
    $\mathfrak{F} = (W, \DmdOp)$ and model $\mathfrak{M} = (\mathfrak{F}, \vartheta)$, their duals
    are as follows: $\mathfrak{F}^d \defeq (W, \BoxOp)$ and
    $\mathfrak{M}^d \defeq (\mathfrak{F}^d, \vartheta)$. For sets of formulas and classes of frames
    and models, their duals are defined in an obvious way. The following properties can be easily
    checked~\cite[Section 10.6 ]{Han03}:
    \begin{enumerate}
      \item \label{d-log} $\Lambda$ is a logic iff $\Lambda^d$ is a logic;
      \item \label{d-deriv} $\Gamma \vdash_\Lambda \varphi$ iff $\Gamma^d \vdash_\Lambda \varphi^d$;
      \item \label{d-truth} $\mathfrak{M}, w \vDash \varphi$ iff
            $\mathfrak{M}^d, w \vDash \varphi^d$ (in other words,
            $\vartheta(\varphi) = \vartheta^d(\varphi^d)$, where $\vartheta^d$ is an extension of
            $\vartheta$ in the dual frame);
      \item \label{d-logOf} $(\mathrm{Log}\, \mathfrak{F})^d = \mathrm{Log}\, \mathfrak{F}^d$;
      \item \label{d-bisim} $Z$ is a $\tau$-bisimulation between models $\mathfrak{M}_1$ and
            $\mathfrak{M}_2$ iff $Z$ is a $\tau$-bisimulation between models $\mathfrak{M}_1^d$ and
            $\mathfrak{M}_2^d$;
    \end{enumerate}
    \begin{proposition}
      $(\mathfrak{M}^\sigma_\Lambda)^d \cong \mathfrak{M}^\pi_{\Lambda^d}$ and
      $(\mathfrak{M}^\pi_\Lambda)^d \cong \mathfrak{M}^\sigma_{\Lambda^d}$, where isomorphism maps
      $w \in W_\Lambda$ into $w^d$.
    \end{proposition}
    \begin{proof}
      By \cref{d-deriv}, $\Gamma$ is $\Lambda$-consistent iff $\Gamma^d$ is $\Lambda^d$ consistent,
      whence ${\cdot}^d : w \mapsto w^d$ is a one-to-one correspondence between $W_\Lambda$ and
      $W_{\Lambda^d}$. Moreover, $\Gamma \subseteq w \Leftrightarrow \Gamma^d \subseteq w^d$, whence
      $\llbracket\Gamma\rrbracket_\Lambda^d \defeq (\llbracket\Gamma\rrbracket_\Lambda)^d =
      \llbracket\Gamma^d\rrbracket_{\Lambda^d}$. In particular,
      $\llbracket p\rrbracket_\Lambda^d = \llbracket p\rrbracket_{\Lambda^d}$, that is, ${\cdot}^d$
      preserves the valuation. It remains to show that,
      \begin{gather*}
        (\BoxOp^\pi_\Lambda X)^d = \DmdOp^\pi_{\Lambda^d} X^d \quad\text{and}\quad
        (\BoxOp^\sigma_\Lambda X)^d = \DmdOp^\sigma_{\Lambda^d} X^d \quad\text{for all }X \subseteq
        W_\Lambda.
      \end{gather*}
      Notice that one can easily derive one of these equalities from another, e.g.
      \begin{gather*}
        (\BoxOp^\pi_\Lambda X)^d = \bigl(-\DmdOp^\sigma_\Lambda(-X)\bigr)^d =
        \bigl(-(\BoxOp^\sigma_{\Lambda^d}(-X)^d)^d\bigr)^d = \DmdOp^\pi_{\Lambda^d}X^d.
      \end{gather*}
      Let us prove the second equality. We have:
      \begin{gather*}
        (\BoxOp^\sigma_\Lambda X)^d = \bigcup_{\Gamma : \llbracket\Gamma\rrbracket_\Lambda \subseteq
        X}\, \bigcap_{\varphi : \llbracket\Gamma\rrbracket_\Lambda \subseteq
        \llbracket\varphi\rrbracket_\Lambda} \llbracket\BoxOp\varphi\rrbracket_\Lambda^d.
      \end{gather*}
      Notice that
      \begin{gather*}
        \llbracket\Gamma\rrbracket_\Lambda \subseteq X \Leftrightarrow
        \llbracket\Gamma\rrbracket_\Lambda^d \subseteq X^d \Leftrightarrow
        \llbracket\Gamma^d\rrbracket_{\Lambda^d} \subseteq X^d,\\
        \llbracket\Gamma\rrbracket_\Lambda \subseteq \llbracket\varphi\rrbracket_\Lambda
        \Leftrightarrow \llbracket\Gamma\rrbracket_\Lambda^d \subseteq
        \llbracket\varphi\rrbracket_\Lambda^d \Leftrightarrow
        \llbracket\Gamma^d\rrbracket_{\Lambda^d} \subseteq
        \llbracket\varphi^d\rrbracket_{\Lambda^d},
      \end{gather*}
      and $\llbracket\BoxOp\varphi\rrbracket_\Lambda^d =
      \llbracket\DmdOp\varphi^d\rrbracket_{\Lambda^d}$. Denoting $\Delta = \Gamma^d$ and
      $\psi = \varphi^d$, we obtain
      \begin{gather*}
        (\BoxOp^\sigma_\Lambda X)^d = \bigcup_{\Delta : \llbracket\Delta\rrbracket_{\Lambda^d}
        \subseteq X^d}\, \bigcap_{\psi : \llbracket\Delta\rrbracket_{\Lambda^d} \subseteq
        \llbracket\psi\rrbracket_{\Lambda^d}} \llbracket\DmdOp\psi\rrbracket_{\Lambda^d}.
      \end{gather*}
      This expression clearly equals $\DmdOp^\sigma_{\Lambda^d} X^d$.
    \end{proof}
    Duality allows us to derive from a statement about monotone logics and algebras its dual: an
    object $O$ has a property $P$ iff the dual object $O^d$ has the dual property $P^d$. In
    particular, we have the following
    \begin{corollary}
      $\Lambda$ is $\sigma$- ($\pi$-)canonical iff $\Lambda^d$ is $\pi$- ($\sigma$-)canonical.
    \end{corollary}
 \section{Basic facts about positive formulas and interpolation}
  \subsection{Core definitions}
    Let us fix tuples of variables $\vec p = (p_i)_{i<n}$, $\bar r = (r_j)_{j<m}$ and a logic
    $\Lambda$. A formula $\varphi(\vec p, \bar r)$ is called \emph{$\vec p$-positive} if
    $\mathrm{lits}(\varphi) \cap \neg\vec p = \emptyset$. $\varphi(\vec p, \bar r)$ is
    \emph{$\vec p$-monotone in $\Lambda$} if
    $\Lambda \vdash \varphi(\vec p, \bar r) \to \varphi(\vec p \vee \vec q, \bar r)$. A formula
    $\psi(\vec p)$ is \emph{positive (monotone in $\Lambda$)} if it is $\vec p$-positive
    ($\vec p$-monotone in $\Lambda$). Notice that this definition of positive formulas is equivalent
    to the definition from~\cref{ss:formulas}.

    Let us consider the following partial order on the set of valuations on a modal algebra
    $\mathfrak{A}$:
    \begin{gather*}
      \vartheta_1 \leq_{\vec p} \vartheta_2 :\Leftrightarrow \vartheta_1|_{\vec p} \leq
      \vartheta_2|_{\vec p} \text{ and } \vartheta_1|_{\mathcal{V} \setminus \vec p} =
      \vartheta_2|_{\mathcal{V} \setminus \vec p}.
    \end{gather*}
    $\varphi$ is \emph{$\vec p$-monotone in $\mathfrak{A}$} if, for all valuations
    $\vartheta_1, \vartheta_2$ on $\mathfrak{A}$ such that $\vartheta_1 \leq_{\vec p} \vartheta_2$,
    $\vartheta_1(\varphi) \leq \vartheta_2(\varphi)$.
    \begin{proposition}
      \label{mon-equiv} Suppose that $\Lambda = \mathrm{Log}\, \mathfrak{C}$ for some class of modal
      algebras $\mathfrak{C}$. Then the following conditions on a formula $\varphi(\vec p, \bar r)$
      are equivalent:
      \begin{enumerate}
        \item $\varphi$ is $\vec p$-monotone in all algebras from $\mathfrak{C}$; \label{it:me-1}
        \item $\varphi_\mathfrak{A}(\cdot, \bar c) : \vec a \mapsto \varphi_\mathfrak{A}(\vec a,
              \bar c)$ is monotone for all $\mathfrak{A} \in \mathfrak{C}$, $\bar c \in A^m$
              \label{it:me-2}
        \item $\varphi$ is $\vec p$-monotone in $\Lambda$. \label{it:me-3}
      \end{enumerate}
    \end{proposition}
    \begin{proof}
      \cref{it:me-1} $\Rightarrow$ \cref{it:me-2}. Suppose that $\varphi$ is $\vec p$-monotone in an
      algebra $\mathfrak{A} \in \mathfrak{C}$, $\vec a, \vec b \in A^n$ are such that
      $\vec a \leq \vec b$, $\bar c \in A^m$. Let us show that
      $\varphi_\mathfrak{A}(\vec a, \bar c) \leq \varphi_\mathfrak{A}(\vec b, \bar c)$. Consider the
      following valuations on $\mathfrak{A}$:
      \begin{itemize}
        \item $\vartheta_1(p_i) \defeq a_i$ and $\vartheta_2(p_i) \defeq b_i$ for all $i < n$;
        \item $\vartheta_1(r_j) = \vartheta_2(r_j) \defeq c_j$ for all $j < m$;
        \item $\vartheta_1(q) = \vartheta_2(q) \defeq \emptyset$ for all
              $q \notin \vec p \cup \bar r$.
      \end{itemize}
      Clearly, $\vartheta_1 \leq_{\vec p} \vartheta_2$. Therefore, $\varphi_\mathfrak{A}(\vec a,
      \bar c) = \vartheta_1(\varphi) \leq \vartheta_2(\varphi) = \varphi_\mathfrak{A}(\vec b, \bar
      c)$.

      \cref{it:me-2} $\Rightarrow$ \cref{it:me-3}. Suppose that \cref{it:me-2} holds. Since
      $\Lambda = \mathrm{Log}\,(\mathfrak{C})$, it is sufficient to show that
      $\mathfrak{A} \vDash \varphi(\vec p, \bar r) \to \varphi(\vec p \vee \vec q, \bar r)$ for all
      $\mathfrak{A} \in \mathfrak{C}$. Let $\vartheta$ be a valuation on
      $\mathfrak{A} \in \mathfrak{C}$. We put $a_i \defeq \vartheta(p_i)$,
      $b_i \defeq \vartheta(p_i \vee q_i)$ for $i < n$, $c_j \defeq \vartheta(r_j)$ for $j < m$.
      Clearly $\vec a \leq \vec b$, whence
      \begin{gather*}
        \vartheta\bigl(\varphi(\vec p, \bar r)\bigr) = \varphi_\mathfrak{A}(\vec a, \bar c) \leq
        \varphi_\mathfrak{A}(\vec b, \bar c) = \vartheta\bigl(\varphi(\vec p \vee \vec q, \bar
        r)\bigr).
      \end{gather*}
      Therefore, $(\mathfrak{A}, \vartheta) \vDash \varphi(\vec p, \bar r) \to \varphi(\vec p \vee
      \vec q, \bar r)$.

      \cref{it:me-3} $\Rightarrow$ \cref{it:me-1}. Suppose that \cref{it:me-3} holds. Let
      $\vartheta_1$ and $\vartheta_2$ be valuations on $\mathfrak{A} \in \mathfrak{C}$ such that
      $\vartheta_1 \leq_{\vec p} \vartheta_2$. Consider the valuation $\vartheta$ on $\mathfrak{A}$
      such that $\vartheta(p_i) = \vartheta_1(p_i)$, $\vartheta(q_i) = \vartheta_2(p_i)$ for
      $i < n$, $\vartheta(r) = \vartheta_1(r) = \vartheta_2(r)$ for $r \notin \vec p \cup \vec q$.
      Then
      \begin{gather*}
        \vartheta_1(\varphi(\vec p, \bar r)) = \vartheta(\varphi(\vec p, \bar r)) \leq
        \vartheta(\varphi(\vec p \vee \vec q, \bar r)) = \vartheta(\varphi(\vec q, \bar r)) =
        \vartheta_2(\varphi(\vec p, \bar r)).
      \end{gather*}
    \end{proof}
    The following fact can be easily checked by induction on construction of a $\vec p$-positive
    formula.
    \begin{proposition}
      \label{pos-are-mon} If $\Lambda$ is monotone, then all $\vec p$-positive formulas are
      $\vec p$-monotone in $\Lambda$. In particular, all positive formulas are monotone in
      $\Lambda$.
    \end{proposition}
    \begin{definition}
      Let $\Lambda$ be a monotone logic. $\Lambda$ has the \emph{Lyndon positivity property (LPP)}
      if every $\vec p$-monotone in $\Lambda$ formula $\varphi(\vec p, \bar r)$ is
      $\Lambda$-equivalent to some $\vec p$-positive formula $\alpha(\vec p, \bar r)$.
    \end{definition}
    Notice that LPP defined here is slightly stronger that LPP as it was defined in \cref{s:intro}:
    we consider formulas with parameters $\bar r$ here. In fact, we will show that this is not
    essential for normal logics (\cref{par-norm}). To investigate this problem more precisely, we
    introduce the following graded versions of LPP:
    \begin{definition}
      Let $n, m \leq \omega$, $\Lambda$ be a monotone logic. We say that $\Lambda$ has
      \emph{${\rm LPP}(n,m)$} if, for all finite $n' \leq n$, $m' \leq m$ and tuples of variables
      $\vec p = (p_i)_{i<n'}$, $\bar r = (r_j)_{j<m'}$, every $\vec p$-monotone in $\Lambda$ formula
      $\varphi(\vec p, \bar r)$ is $\Lambda$-equivalent to some $\vec p$-positive formula
      $\alpha(\vec p, \bar r)$.
    \end{definition}
    Notice that ${\rm LPP}(\omega,\omega)$ is the full LPP, how it is defined above,
    ${\rm LPP}(\omega,0)$ is its non-parametric version from \cref{s:intro}, and
    ${\rm LPP}(0,\omega)$ trivially holds in all logics.
    \begin{lemma}
      Let $n, m < \omega$, $\vec p = (p_i)_{i<n}$, $\bar r = (r_j)_{j<m}$. Suppose that every
      $\vec p$-monotone in $\Lambda$ formula $\varphi(\vec p, \bar r)$ is $\Lambda$-equivalent to
      some $\vec p$-positive formula $\alpha(\vec p, \bar r)$. Then $\Lambda$ has ${\rm LPP}(n,m)$.
    \end{lemma}
    \begin{proof}
      Let $\varphi(\vec p', \bar r')$ be a $\vec p'$-monotone in $\Lambda$ formula, where $n' < n$,
      $m' < m$, $\vec p' = (p_i)_{i<n'}$, $\bar r' = (r_j)_{j<m'}$. By the condition of the lemma,
      $\varphi(\vec p', \bar r') \sim_\Lambda \alpha(\vec p, \bar r)$ for some $\vec p$-monotone
      formula $\alpha$. Let $\alpha'(\vec p', \bar r')$ the result of substitution of $\bot$ in
      $\alpha$ for all $p_i, n' \leq i < n$ and all $r_j, m' \leq j < m$. Clearly $\alpha'$ is
      $\vec p'$-positive. Since $\Lambda$ is closed under substitution rule,
      $\varphi \sim_\Lambda \alpha'$.
    \end{proof}
    \begin{proposition}
      \label{lpp-mono} Let $\Lambda$ be a logic with LIP. Then, for all $n, m < \omega$,
      ${\rm LPP}(1,n+m-1) \Rightarrow {\rm LPP}(n,m)$ in $\Lambda$.
    \end{proposition}
    \begin{proof}
      Suppose that $\Lambda$ has ${\rm LPP}(1,s-1)$. We prove that it has ${\rm LPP}(n,s-n)$ for
      $n = 1, \dots, s$ by induction on $n$. The base case is trivial. Suppose that $1 < n \leq s$
      and ${\rm LPP}(n-1,s-n+1)$ holds. Let $\vec p = (p_i)_{i<n-1}$, $\bar r = (r_j)_{j<s-n}$,
      $\varphi(\vec p, q, \bar r)$ be a $(\vec p \cup \{q\})$-monotone in $\Lambda$ formula. By the
      induction hypothesis, there is a $\vec p$-positive formula $\alpha(\vec p, q, \bar r)$ such
      that $\varphi \sim_\Lambda \alpha$. By ${\rm LPP}(1,s-1)$, there is a $q$-positive formula
      $\beta(\vec p, q, \bar r)$ such that $\varphi \sim_\Lambda \beta$. Since
      $\alpha \preceq_\Lambda \beta$, there is a Lyndon interpolant $\iota$ for $\alpha$ and
      $\beta$. Clearly, $\iota$ is $(\vec p \cup \{q\})$-positive and $\iota \sim_\Lambda \varphi$.
    \end{proof}
    \begin{corollary}
      If $\Lambda$ has LIP and ${\rm LPP}(1,\omega)$, then $\Lambda$ has LPP.
    \end{corollary}
  \subsection{Generalized interpolation}
    \begin{definition}
      Let $\tau \subseteq \mathcal{V}^\pm$, $\varphi, \psi \in \mathrm{Fm}$. We say that a formula
      $\iota$ is a \emph{$\tau$-interpolant} for $\varphi$ and $\psi$ if
      $\mathrm{lits}(\iota) \subseteq \tau$ and
      $\varphi \preceq_\Lambda \iota \preceq_\Lambda \psi$.
    \end{definition}
    Notice that
    \begin{itemize}
      \item $\iota$ is a Lyndon interpolant for $\varphi$ and $\psi$ iff $\iota$ is a
            $\tau$-interpolant for these formulas, where
            $\tau \defeq \mathrm{lits}(\varphi) \cap \mathrm{lits}(\psi)$.
      \item $\iota$ is a Craig interpolant for $\varphi$ and $\psi$ iff $\iota$ is a
            $\tau$-interpolant for these formulas, where
            $\tau \defeq \mathrm{vars}(\varphi)^\pm \cap \mathrm{vars}(\psi)^\pm$.
      \item $\iota$ is a $\vec p$-positive formula which is $\Lambda$-equivalent to $\varphi$ iff
            $\iota$ is a $\tau$-interpolant for $\varphi$ and $\varphi$, where
            $\tau \defeq \nu^\pm \setminus \neg\vec p$.
    \end{itemize}
    In fact, $\sigma^\pm$-interpolants for $\sigma \subseteq \mathcal{V}$ can be derived from Craig
    interpolants by the following trick:
    \begin{proposition}
      \label{gen-craig-trick} Let $\varphi$ and $\psi(\vec p, \bar q)$ be formulas, $\bar r$ be a
      tuple of fresh variables. Then $\iota$ is a $\vec p^\pm$-interpolant for $\varphi$ and $\psi$
      iff $\iota$ is a Craig interpolant for $\varphi \wedge \eta$ and
      $\psi(\vec p, \bar r) \wedge \eta$, where $\eta \defeq \bigvee_{i<n}p_i \vee \top$.
    \end{proposition}
    \begin{proof}
      Clearly, $\mathrm{vars}(\varphi \wedge \eta)^\pm \cap \mathrm{vars}(\psi(\vec p, \bar r)
      \wedge \eta) = \vec p$. Notice that $\varphi \wedge \eta$ and
      $\psi(\vec p, \bar r) \wedge \eta$ are equivalent to $\varphi$ and $\psi(\vec p, \bar r)$
      respectively. Moreover, since neither $\bar q$ nor $\bar r$ occur in $\iota$,
      \begin{gather*}
        \iota(\vec p) \preceq_\Lambda \psi(\vec p, \bar q) \Leftrightarrow \iota(\vec p)
        \preceq_\Lambda \psi(\vec p, \bar r)
      \end{gather*}
      by the substitution rule.
    \end{proof}
    However, it is unclear wether we can derive $\tau$-interpolants from Lyndon interpolants in a
    similar way.
  \subsection{Interpolation and canonical frames}
    Results from these section are variants of well-known results from classical model theory and
    modal logic (see, e.g.,~\cite[Chapter 8]{Hodg93} and~\cite[Chapter 5, Section 3.7]{HoML}).
    \begin{lemma}
      Let $I$ be a set of formulas closed under $\vee$ and $\wedge$. Then, for each
      $\varphi, \psi \in \mathrm{Fm}$, exactly one of the following holds:
      \begin{itemize}
        \item there is $i \in I$ such that $\varphi \preceq_\Lambda \iota \preceq_\Lambda \psi$;
        \item there are $w \in \llbracket\varphi\rrbracket$ and $v \in \llbracket\neg\psi\rrbracket$
              such that $w \cap I \subseteq v$.
      \end{itemize}
    \end{lemma}
    \begin{proof}
      It is easy to see that both conditions can not hold at once. Indeed, if
      $\varphi \preceq_\Lambda \iota \preceq_\Lambda \psi$ for some $i \in I$,
      $w \in \llbracket\varphi\rrbracket$, and $v \in \llbracket\neg\psi\rrbracket$, then
      $\iota \in w$ and $\iota \notin v$, whence $w \cap I \nsubseteq v$. Therefore, it is
      sufficient to show that at least one of the conditions must be true. Consider the following
      set of formulas:
      \begin{gather*}
        A \defeq \{\alpha \in I \mid \varphi \preceq_\Lambda \alpha\}.
      \end{gather*}
      If $A \cup \{\neg\psi\}$ is $\Lambda$-inconsistent, then there is a finite $A_0 \subseteq A$
      such that ${A_0 \vdash_\Lambda \psi}$. Let $\iota \defeq \bigwedge A_0$. Clearly,
      $\varphi \preceq_\Lambda \iota \preceq_\Lambda \psi$. Since $I$ is closed under conjuctions,
      $\iota \in I$.

      Now suppose that $A \cup \{\neg\psi\}$ is $\Lambda$-consistent. Then, by Lindenbaum's lemma,
      there is $v \in \llbracket\neg\psi\rrbracket$ such that $A \subseteq v$. Let
      $B \defeq I \setminus v$.

      If $\neg B \cup \{\varphi\}$ is $\Lambda$-inconsistent, then there is a finite
      $B_0 \subseteq B$ such that $\varphi \preceq_\Lambda \bigvee B_0$. Since $I$ is closed under
      disjunctions, $\bigvee B_0 \in I$. Therefore, $\bigvee B_0 \in A \subseteq v$, whence there is
      $\beta \in B_0$ such that $\beta \in v$. But $B_0 \subseteq B = I \setminus v$. Contradiction.

      Thus, $\neg B \cup \{\varphi\}$ is $\Lambda$-consistent and, by Lindenbaum's lemma, there is
      $w \in \llbracket\varphi\rrbracket$ such that $w \cap B = \emptyset$. Clearly,
      $w \cap I \subseteq I \setminus B \subseteq v$.
    \end{proof}
    For a set of literals $\tau$, consider the preorder on $W_\Lambda$:
    \begin{gather*}
      w \trianglelefteq_\tau v :\Leftrightarrow w \cap \mathrm{Fm}_\tau \subseteq v.
    \end{gather*}
    Notice that $\trianglelefteq_\tau^{-1} = \trianglelefteq_{\neg\tau}$. In particular,
    $\trianglelefteq_{\sigma^\pm}$ for $\sigma \subseteq \mathcal{V}$ is in fact an equivalence
    relation.
    \begin{corollary}
      \label{interp-stone} Let $\tau$ be a set of literals. Then, for each
      $\varphi, \psi \in \mathrm{Fm}$, exactly one of the following holds:
      \begin{itemize}
        \item there is a $\tau$-interpolant for $\varphi$ and $\psi$ in $\Lambda$;
        \item there are $w \in \llbracket\varphi\rrbracket$ and $v \in \llbracket\neg\psi\rrbracket$
              such that $w \trianglelefteq_\tau v$.
      \end{itemize}
    \end{corollary}
    \begin{proposition}
      \label{can-sim} $\trianglelefteq_\tau$ is a full $\tau$-bisimulation on both
      $\mathfrak{M}^\sigma_\Lambda$ and $\mathfrak{M}^\pi_\Lambda$.
    \end{proposition}
    \begin{proof}
      By duality, it is sufficient to consider $\mathfrak{M}^\pi_\Lambda$. It is clear that
      $\trianglelefteq_\tau$ is full and the condition $({\rm lit})$ is satisfied. Since
      ${\trianglelefteq_\tau^{-1}} = {\trianglelefteq_{\neg\tau}}$, $({\rm zag}')$ for
      $\trianglelefteq_\tau$ follows from $({\rm zig})$ for $\trianglelefteq_{\neg\tau}$. Thus, it
      is sufficient to show that
      \begin{gather*}
        \A{ X \subseteq W_\Lambda} \bigl({\trianglelefteq_\tau}(\DmdOp^\sigma_\Lambda X) \subseteq
        \DmdOp^\sigma_\Lambda({\trianglelefteq_\tau} X)\bigr).
      \end{gather*}
      Also notice that, since
      \begin{gather*}
        {\trianglelefteq_\tau}(\DmdOp^\sigma_\Lambda X) =
        {\trianglelefteq_\tau}\left(\bigcup_{\Gamma : \llbracket\Gamma\rrbracket \subseteq
        X}\DmdOp^\sigma_\Lambda\llbracket\Gamma\rrbracket\right) = \bigcup_{\Gamma :
        \llbracket\Gamma\rrbracket \subseteq
        X}{\trianglelefteq_\tau}(\DmdOp^\sigma_\Lambda\llbracket\Gamma\rrbracket)
      \end{gather*}
      and, for $\llbracket\Gamma\rrbracket \subseteq X$,
      $\DmdOp^\sigma_\Lambda({\trianglelefteq_\tau} \llbracket\Gamma\rrbracket) \subseteq
      \DmdOp^\sigma_\Lambda({\trianglelefteq_\tau} X)$, we can assume that
      $X = \llbracket\Gamma\rrbracket$.

      Suppose that
      $w_2 \in {\trianglelefteq_\tau}(\DmdOp^\sigma_\Lambda \llbracket\Gamma\rrbracket)$, that is,
      there is $w_1 \in \DmdOp^\sigma_\Lambda\llbracket\Gamma\rrbracket$ such that
      $w_1 \trianglelefteq_\tau w_2$. We need to show that
      $w_2 \in \DmdOp^\sigma_\Lambda({\trianglelefteq_\tau} \llbracket\Gamma\rrbracket)$. Consider
      the set of formulas
      $\Delta \defeq \{\psi \in \mathrm{Fm}_\tau \mid \Gamma \vdash_\Lambda \psi\}$. We will show
      that $w_2 \in \DmdOp^\sigma_\Lambda\llbracket\Delta\rrbracket$ and
      $\llbracket\Delta\rrbracket \subseteq {\trianglelefteq_\tau}\llbracket\Gamma\rrbracket$. Then,
      by monotonicity,
      $w_2 \in \DmdOp^\sigma_\Lambda({\trianglelefteq_\tau}\llbracket\Gamma\rrbracket)$.

      To prove that $w_2 \in \DmdOp^\sigma_\Lambda\llbracket\Delta\rrbracket =
      \llbracket\{\DmdOp\varphi \mid \Delta \vdash_\Lambda \varphi\}\rrbracket$, we fix some
      $\varphi$ such that $\Delta \vdash_\Lambda \varphi$. Then, since $\Delta$ is closed under
      conjunctions, there is $\psi \in \Delta$ such that $\psi \preceq_\Lambda \varphi$. Notice that
      $w_1 \in \DmdOp^\sigma_\Lambda\llbracket\Gamma\rrbracket \subseteq
      \llbracket\DmdOp\psi\rrbracket$. Since $w_1 \trianglelefteq_\tau w_2$ and
      $\DmdOp\psi \in \mathrm{Fm}_\tau$,
      $w_2 \in \llbracket\DmdOp\psi\rrbracket \subseteq \llbracket\DmdOp\varphi\rrbracket$.

      To prove that
      $\llbracket\Delta\rrbracket \subseteq {\trianglelefteq_\tau}\llbracket\Gamma\rrbracket$, we
      fix some $v_2 \in \llbracket\Delta\rrbracket$. Then we put
      $\Sigma \defeq \mathrm{Fm}_\tau \setminus v_2$ and consider the set $\Gamma \cup \neg\Sigma$.
      If $\Gamma \cup \neg\Sigma$ is $\Lambda$-consistent, then, by Lindenbaum's lemma, there is
      $v_1 \supseteq \Gamma \cup \neg\Sigma$. Clearly, $v_1 \in \llbracket\Gamma\rrbracket$ and
      $v_1 \trianglelefteq_\tau v_2$, whence
      $v_2 \in {\trianglelefteq_\tau}\llbracket\Gamma\rrbracket$. Now, suppose that
      $\Gamma \cup \neg\Sigma$ is $\Lambda$-inconsistent. Then, since $\Sigma$ is closed under
      disjunctions, there is $\chi \in \Sigma$ such that $\Gamma \vdash_\Lambda \chi$. Since
      $\Sigma \subseteq \mathrm{Fm}_\tau$, $\chi \in \Delta \subseteq v_2$. But
      $\Sigma \cap v_2 = \emptyset$. Contradiction.
    \end{proof}

    Let $\mathfrak{M}_k = (W_k, \DmdOp_k, \vartheta_k)$, $k = 1,2$ be models,
    $Z \subseteq W_1 \times W_2$, $\varphi, \psi \in \mathrm{Fm}$. We say that \emph{$\varphi$
    entails $\psi$ under $Z$} if
    \begin{gather*}
      \A{(w_1, w_2) \in Z} (\mathcal{M}_1, w_1 \vDash \varphi \Rightarrow \mathcal{M}_2, w_2 \vDash
      \psi).
    \end{gather*}
    Notice that $\varphi$ is preserved under $Z$ in the sense of the definition from \cref{ss:bisim}
    iff $\varphi$ entails itself under $Z$.
    \begin{corollary}
      \label{lam-char} Let $\Lambda$ be a monotone logic, $\tau$ be a set of literals. Then the
      following conditions on $\varphi, \psi \in \mathrm{Fm}$ are equivalent:
      \begin{enumerate}
        \item there is a $\tau$-interpolant for $\varphi$ and $\psi$ in $\Lambda$; \label{it:lc-1}
        \item $\varphi$ entails $\psi$ under all $\tau$-bisimulations between $\Lambda$-models;
              \label{it:lc-2}
        \item $\varphi$ entails $\psi$ under $\trianglelefteq_\tau$ on $\mathfrak{M}_\Lambda^\sigma$
              (or on $\mathfrak{M}_\Lambda^\pi$). \label{it:lc-3}
      \end{enumerate}
    \end{corollary}

    \begin{proof}
      \cref{it:lc-1} $\Rightarrow$ \cref{it:lc-2} follows from \cref{bisim-main}.

      \cref{it:lc-2} $\Rightarrow$ \cref{it:lc-3} follows from \cref{can-sim}.

      \cref{it:lc-3} $\Rightarrow$ \cref{it:lc-1} follows from \cref{interp-stone}.
    \end{proof}
  \subsection{Preservation theorems}
    In the classical predicate calculus, LPP is closely connected with the following property, known
    as Lyndon's preservation theorem~\cite[Corollary 5.3]{Lynd59Hom}: a formula is positive iff it
    is preserved under surjective homomorphisms. So, in CPC, the following conditions on a formula
    $\varphi$ are equivalent:
    \begin{itemize}
      \item $\varphi$ is equivalent to a positive formula;
      \item $\varphi$ is monotone;
      \item $\varphi$ is preserved under surjective homomorphisms.
    \end{itemize}
    In modal logics, as we will show later, LPP does not hold in general. So, we need two different
    model-theoretic characterizations: one for positive formulas and another one for monotone. The
    following theorem is a generalization of~\cite[Corollary 3.9]{KurtDRij97} to the case of
    non-normal logics.
    \begin{theorem}
      \label{pos-char} Let $\Lambda$ be a monotone logic, $\vec p$ be a tuple of variables. The
      following conditions on a formula $\varphi$ are equivalent:
      \begin{enumerate}
        \item $\varphi$ is $\Lambda$-equivalent to a $\vec p$-positive formula;
        \item $\varphi$ is preserved under all $\vec p$-directed bisimulations between
              $\Lambda$-models;
        \item $\varphi$ is preserved under $\trianglelefteq_{\vec p}$ on
              $\mathfrak{M}_\Lambda^\sigma$ (or on $\mathfrak{M}_\Lambda^\pi$).
      \end{enumerate}
    \end{theorem}
    \begin{proof}
      This theorem follows immediately from \cref{lam-char}.
    \end{proof}
    We say that a relation $Z \subseteq W_1 \times W_2$ is \emph{zigzag-free} if it can be presented
    in the form $Z = Z_1 \cup Z_2$, where $Z_1$ and $Z_2^{-1}$ are functions and
    \begin{gather}
      {\rm dom}\,(Z_1) \cap {\rm dom}\,(Z_2) = {\rm rng}\,(Z_1) \cap {\rm rng}\,(Z_2) = \emptyset.
      \label{eq:272}
    \end{gather}
    \vspace{-4ex}
    \begin{lemma}
      \label{l:939} Suppose that $\mathfrak{F}_1 = (W_1, \DmdOp_1)$ and
      $\mathfrak{F}_2 = (W_2, \DmdOp_2)$ are $\Lambda$-frames, $Z$ is a zigzag-free
      $\vec p$-directed bisimulation between $\mathfrak{M}_1 = (\mathfrak{F}_1, \vartheta_1)$ and
      $\mathfrak{M}_2 = (\mathfrak{F}_2, \vartheta_2)$. Then all $\vec p$-monotone in $\Lambda$
      formulas are preserved under $Z$.
    \end{lemma}
    \begin{proof}
      Let $Z_1$ and $Z_2^{-1}$ be functions such that $Z = Z_1 \cup Z_2$ and \cref{eq:272} holds.
      Consider the following valuations:
      \begin{itemize}
        \item $\vartheta'_1$ on $\mathfrak{F}_1$ such that $\vartheta'_1(q) \defeq
              Z_1^{-1}\vartheta_2(q) \cup \vartheta_1(q) \setminus {\rm dom}\, Z_1$ for all
              $q \in \mathcal{V}$;
        \item $\vartheta'_2$ on $\mathfrak{F}_2$ such that $\vartheta'_2(q) \defeq Z_2
              \vartheta_1(q) \cup \vartheta_2(q) \setminus {\rm rng}\, Z_2$ for all
              $q \in \mathcal{V}$.
      \end{itemize}
      Let us check that $\vartheta_1 \leq_{\vec p} \vartheta_1'$ and
      $\vartheta_2' \leq_{\vec p} \vartheta_2$. For all $q \in \mathcal{V}$,
      \begin{itemize}
        \item $\vartheta_1(q) \cap {\rm dom}\, Z_1 \subseteq Z_1^{-1}\vartheta_2(q)$, whence
              $\vartheta_1(q) \subseteq \vartheta_1'(q)$;
        \item $Z_2 \vartheta_1(q) \subseteq \vartheta_2(q)$, whence
              $\vartheta_2'(q) \subseteq \vartheta_2(q)$.
      \end{itemize}
      For $r \in \mathcal{V} \setminus \vec p$,
      \begin{itemize}
        \item $\vartheta_1(r) \supseteq Z_1^{-1}\vartheta_2(r)$, whence
              $\vartheta_1(r) \supseteq \vartheta_1'(r)$;
        \item $Z_2 \vartheta_1(r) \supseteq \vartheta_2(r) \cap {\rm rng}\, Z_2$, whence
              $\vartheta_2'(r) \supseteq \vartheta_2(r)$.
      \end{itemize}
      Moreover, $Z$ is a bisimulation between $(\mathfrak{F}_1, \vartheta_1')$ and
      $(\mathfrak{F}_2, \vartheta_2')$. Indeed, for $q \in \mathcal{V}$, $(w_1, w_2) \in Z_1$,
      \begin{gather*}
        w_1 \in \vartheta'_1(q) \Leftrightarrow w_1 \in Z_1^{-1}\vartheta_2(q) \Leftrightarrow w_2
        \in \vartheta_2(q) \Leftrightarrow w_2 \in \vartheta'_2(q),
      \end{gather*}
      and, for $(w_1, w_2) \in Z_2$,
      \begin{gather*}
        w_1 \in \vartheta'_1(q) \Leftrightarrow w_1 \in \vartheta_1(q) \Leftrightarrow w_2 \in Z_2
        \vartheta_1(q) \Leftrightarrow w_2 \in \vartheta'_2(q).
      \end{gather*}
      Now, one can easily finish the proof: if $\varphi$ is $\vec p$-monotone in $\Lambda$, then,
      since $\mathfrak{F}_1$ and $\mathfrak{F}_2$ are $\Lambda$-frames,
      $\vartheta_1(\varphi) \subseteq \vartheta'_1(\varphi)$ and
      $\vartheta'_2(\varphi) \subseteq \vartheta_2(\varphi)$. Since $Z$ is a bisimulation,
      $Z(\vartheta_1'(\varphi)) \subseteq \vartheta_2'(\varphi)$. Thus,
      $Z(\vartheta_1(\varphi)) \subseteq \vartheta_2(\varphi)$.
    \end{proof}
    \begin{proposition}
      \label{mon-char} Suppose that $\Lambda = \mathrm{Log}\, \mathfrak{C}$, where $\mathfrak{C}$ is
      a class of neighborhood frames. Then the following conditions on a formula $\varphi$ are
      equivalent:
      \begin{enumerate}
        \item $\varphi$ is $\vec p$-monotone in $\Lambda$; \label{it:mc-1}
        \item $\varphi$ is preserved under zigzag-free $\vec p$-directed bisimulations between the
              models based on $\Lambda$-frames; \label{it:mc-2}
        \item $\varphi$ is preserved under $\vec p$-directed morphisms between the models based on
              the frames from $\mathfrak{C}$; \label{it:mc-3}
        \item $\varphi$ is $\vec p$-monotone in all frames from $\mathfrak{C}$. \label{it:mc-4}
      \end{enumerate}
    \end{proposition}
    \begin{proof}
      \cref{it:mc-1} $\Rightarrow$ \cref{it:mc-2} follows from \cref{l:939}. \cref{it:mc-2}
      $\Rightarrow$ \cref{it:mc-3} is trivial.

      \cref{it:mc-3} $\Rightarrow$ \cref{it:mc-4}. Suppose that \cref{it:mc-3} holds. Let
      $\vartheta_1$ and $\vartheta_2$ be valuations on $\mathfrak{F} = (W, \DmdOp) \in \mathfrak{C}$
      such that $\vartheta_1 \leq_{\vec p} \vartheta_2$. Then $1_W$ is a $\vec p$-directed morphism
      from $(\mathfrak{F}, \vartheta_1)$ to $(\mathfrak{F}, \vartheta_2)$. Therefore,
      $\vartheta_1(\varphi) = 1_W(\vartheta_1(\varphi)) \subseteq \vartheta_2(\varphi)$.

      \cref{it:mc-4} $\Rightarrow$ \cref{it:mc-1} follows from \cref{mon-equiv}.
    \end{proof}
    Notice that \cref{mon-char} gives a characterization of monotone formulas only for
    neighborhood-complete logics. In fact, one can obtain a similar result for arbitrary logics and
    general neighborhood frames.
 \section{Monotonicity vs positivity in normal logics}
  \subsection{Lyndon's positivity theorem}
    For a formula $\varphi$ and $n \in \omega$, we denote
    \begin{gather*}
      \DmdOp^{\leq n}\varphi \defeq \bigvee_{k \leq n}\DmdOp^k \varphi,\quad \BoxOp^{\leq n}\varphi
      \defeq \bigwedge_{k \leq n}\BoxOp^k\varphi.
    \end{gather*}
    As usual, the modal depth $d(\varphi)$ of a formula $\varphi$ is defined inductively:
    \begin{align*}
      &d(\alpha) \defeq 0, &&\alpha \in \mathcal{V}^{\pm} \cup \{\bot, \top\},\\
      &d(\varphi \circ \psi) \defeq \mathrm{max}\{ d(\varphi), d(\psi) \}, &&\varphi, \psi \in
      \mathrm{Fm},\; {\circ} \in \{\vee, \wedge\},\\
      &d(\#\varphi) \defeq d(\varphi) + 1, &&\varphi \in \mathrm{Fm}, \# \in \{\DmdOp, \BoxOp\}.
    \end{align*}
    \vspace{-4ex}
    \begin{lemma}
      \label{l:equiv} Let $\vec p = (p_i)_{i < n}$, $\varphi(\vec p)$ be a formula of modal depth
      $d$, $\vec \eta = (\eta_i)_{i < n}$ be a tuple of formulas. Then
      \begin{gather*}
        \mathrm{K} \vdash \BoxOp^{\leq d}\bigwedge_{i<n}(p_i \leftrightarrow \eta_i) \to
        \bigl(\varphi(\vec p) \leftrightarrow \varphi(\vec \eta)\bigr).
      \end{gather*}
    \end{lemma}
    \begin{proof}
      We proceed by induction on construction of $\varphi$. Consider the case $\varphi = \#\psi$,
      where $\# \in \{\DmdOp, \BoxOp\}$. By the induction hypothesis
      \begin{gather*}
        \mathrm{K} \vdash \BoxOp^{\leq d-1}\bigwedge_{i<n}(p_i \leftrightarrow \eta_i) \to
        \bigl(\psi(\vec p) \leftrightarrow \psi(\vec \eta)\bigr).
      \end{gather*}
      By normalitsy,
      \begin{gather*}
        \mathrm{K} \vdash \BoxOp^{\leq d}\bigwedge_{i<n}(p_i \leftrightarrow \eta_i) \to
        \BoxOp\bigl(\psi(\vec p) \leftrightarrow \psi(\vec \eta)\bigr).
      \end{gather*}
      It remains to notice that $\mathrm{K} \vdash \BoxOp\bigl(\psi(\vec p) \leftrightarrow
      \psi(\vec \eta)\bigr) \to \bigl(\#\psi(\vec p) \leftrightarrow \#\psi(\vec \eta)\bigr)$.
    \end{proof}
    \begin{proposition}
      Let $\Lambda$ be a normal logic with ${\rm LPP}(n+2m,0)$. Then $\Lambda$ has ${\rm LPP}(n,m)$.
    \end{proposition}
    \begin{proof}
      Let $\vec p = (p_i)_{i<n}$, $\bar r = (r_j)_{j<m}$, $\varphi(\vec p, \bar r)$ be a
      $\vec p$-monotone in $\Lambda$ formula. By \cref{l:nnf}, there is a formula
      $\varphi'(\vec p, \bar r, \bar r')$ such that $\varphi'$ is $(\bar r \cup \bar r')$-positive
      and $\varphi(\vec p, \bar r) = \varphi'(\vec p, \bar r, \neg\bar r)$. We denote by $d$ the
      modal depth of $\varphi$. Consider the formula
      \begin{gather*}
        \psi(\vec p, \bar r, \bar r') \defeq \varphi'(\vec p, \bar r, \bar r') \wedge
        \bigwedge_{j<m}\BoxOp^{\leq d}(r_j \vee r_j')
        \vee \bigvee_{j<m}\DmdOp^{\leq d}(r_j \wedge r_j').
      \end{gather*}
      Notice that
      \begin{align*}
        \psi(\vec p, \bar r, \bar r')
        &\sim_\mathrm{E} \varphi' \wedge
        \bigwedge_{j<m}\BoxOp^{\leq d}(r_j \vee r_j') \wedge \bigwedge_{j<m}\neg\DmdOp^{\leq d}(r_j
        \wedge r_j')
        \vee \bigvee_{j<m}\DmdOp^{\leq d}(r_j \wedge r_j')\\
        &\sim_\mathrm{K} \varphi' \wedge \bigwedge_{j<m}\BoxOp^{\leq d}(r_j' \leftrightarrow
        \neg r_j)
        \vee \bigvee_{j<m}\DmdOp^{\leq d}(r_j \wedge r_j')\\
        &\sim_\mathrm{K} \varphi'(\vec p, \bar r, \neg\bar r) \vee \bigvee_{j<m}\DmdOp^{\leq d}(r_j
        \wedge r_j')
      \end{align*}
      by \cref{l:equiv}. Since $\varphi'(\vec p, \bar r, \neg\bar r) = \varphi(\vec p, \bar r)$ is
      $\vec p$-monotone in $\Lambda$, $\psi(\vec p, \bar r, \bar r')$ is also $\vec p$-monotone in
      $\Lambda$. Moreover, $\psi(\vec p, \bar r, \bar r')$ is $(\bar r \cup \bar r')$-positive.
      Therefore, $\psi$ is monotone in $\Lambda$ and, by ${\rm LPP}(n+2m,0)$, there is a positive
      formula $\eta(\vec p, \bar r, \bar r')$ such that
      $\eta(\vec p, \bar r, \bar r') \sim_\Lambda \psi(\vec p, \bar r, \bar r')$. Notice that
      \begin{gather*}
        \psi(\vec p, \bar r, \neg\bar r) \sim_\mathrm{K} \varphi'(\vec p, \bar r, \neg\bar r) =
        \varphi(\vec p, \bar r).
      \end{gather*}
      Thus, $\eta(\vec p, \bar r, \neg\bar r)$ is a $\vec p$-positive formula which is
      $\Lambda$-equivalent to $\varphi(\vec p, \bar r)$.
    \end{proof}
    \begin{corollary}
      \label{par-norm} If $\Lambda$ is a normal logic with ${\rm LPP}(\omega,0)$, then it has LPP.
    \end{corollary}
    \begin{theorem}
      \label{lynd} Let $\Lambda$ be a normal modal logic with LIP. Then $\Lambda$ has LPP.
    \end{theorem}
    \begin{proof}
      By \cref{par-norm}, it is sufficient to prove ${\rm LPP}(\omega,0)$. Suppose that a formula
      $\varphi(\vec p)$ is monotone in $\Lambda$. Let $d$ be the modal depth of $\varphi$. By
      \cref{l:equiv},
      \begin{gather*}
        \mathrm{K} \vdash \bigwedge_{i < n}\BoxOp^{\leq d}\bigl((p_i \vee q_i) \leftrightarrow
        p_i\bigr) \to \bigl(\varphi(\vec p \vee \vec q) \leftrightarrow \varphi(\vec p)\bigr).
      \end{gather*}
      Therefore, since $\Lambda \vdash \varphi(\vec q) \to \varphi(\vec p \vee \vec q)$,
      \begin{gather*}
        \Lambda \vdash \bigwedge_{i < n}\BoxOp^{\leq d}(q_i \to p_i) \to \bigl(\varphi(\vec q) \to
        \varphi(\vec p)\bigr),
      \end{gather*}
      that is, $\Lambda \vdash \eta(\vec p, \vec q) \to \varphi(\vec p)$, where
      \begin{gather*}
        \eta(\vec p, \vec q) = \bigwedge_{i < n}\BoxOp^{\leq d}(\neg q_i \vee p_i) \wedge
        \varphi(\vec q).
      \end{gather*}
      Since $\mathrm{lits}(\eta(\vec p, \vec q)) \subseteq \vec p \cup \vec q^\pm$ and
      $\mathrm{lits}(\varphi(\vec p)) \subseteq \vec p^\pm$, by LIP, there is a positive formula
      $\alpha(\vec p)$ such that
      \begin{gather*}
        \eta(\vec p, \vec q) \preceq_\Lambda \alpha(\vec p) \preceq_\Lambda \varphi(\vec p).
      \end{gather*}
      Substituting $\vec p$ for $\vec q$, we obtain
      \begin{gather*}
        \varphi(\vec p) \sim_\Lambda \eta(\vec p, \vec p) \preceq_\Lambda \alpha(\vec p)
        \preceq_\Lambda \varphi(\vec p).
      \end{gather*}
      Thus, $\varphi \sim_\Lambda \alpha$.
    \end{proof}
    Notice that there are two conditions on $\Lambda$ in~\cref{lynd}:
    \begin{itemize}
      \item $\Lambda$ must be normal. This condition is necessary for \cref{l:equiv}, but we do not
            know if it is actually essential for \cref{lynd}, that is, if LIP $\Rightarrow$ LPP for
            non-normal logics. We will show that normality is not necessary for LPP: all canonical
            monotone logics which are preserved under bisimulation products have this property.
      \item $\Lambda$ must have LIP. We will show that this condition is essential. For example, all
            logics between $\mathrm{K4}.3$ and $\mathrm{S4}.3$ lack LPP. At the same time it is not
            necessary: there are infinitely many tabular extensions of $\mathrm{S4}$ without LIP
            (and even CIP) in which LPP holds. However, it is an interesting open question whether
            CIP and (some strengthened version of) LPP imply LIP.
    \end{itemize}
  \subsection{Monotonicity vs positivity in the logics of linear frames}
    It is well-known that all extensions of $\mathrm{K4}.3$ with unbounded height lack
    CIP~\cite[Proposition 10.23]{GabMaks05}. Here, we will show that LPP in all logics between
    $\mathrm{K4}.3$ and $\mathrm{S4}.3$. For other logics of linear frames of unbounded height such
    as $\mathrm{GL}.3$ and $\mathrm{Grz}.3$, LPP is left open (see also a note at the end of this
    section).

    Let us fix variables $p$, $r$, and $s$. We will use the following abbreviations:
    $\DmdOp_+\eta \defeq \DmdOp(\neg s \wedge r \wedge \eta)$,
    $\DmdOp_-\eta \defeq \DmdOp(\neg s \wedge \neg r \wedge \eta)$. Consider the formula
    \begin{gather*}
      \varphi(p, r, s) \defeq \DmdOp_-\DmdOp_+\neg p \wedge \BoxOp_+\BoxOp_-\BoxOp_+ p \vee
      \BoxOp_+ p.
    \end{gather*}
    \vspace{-4ex}
    \begin{lemma}
      \label{l:lin-1} $\varphi(p, r, s)$ is $p$-monotone in $\mathrm{K4}.3$
    \end{lemma}
    \begin{proof}
      Let $\mathcal{F} = (W, R)$ be a $\mathrm{K4}.3$-frame, $w \in W$,
      $\mathcal{M}_1 = (\mathcal{F}, \vartheta_1)$, and
      $\mathcal{M}_2 = (\mathcal{F}, \vartheta_2)$. Suppose that $\vartheta_1 \leq_p \vartheta_2$,
      $\mathcal{M}_1, w \vDash \varphi$, and $\mathcal{M}_2, w \nvDash \varphi$. We denote
      \begin{gather*}
        V_+ \defeq \vartheta_1(r \wedge \neg s) = \vartheta_2(r \wedge \neg s),\quad V_- \defeq
        \vartheta_1(\neg r \wedge \neg s) = \vartheta_2(\neg r \wedge \neg s).
      \end{gather*}
      Since $\BoxOp_+\BoxOp_-\BoxOp_+ p$ and $\BoxOp_+ p$ are $p$-positive, it is easy to see that
      \begin{itemize}
        \item $\DmdOp_-\DmdOp_+ \neg p$ and $\BoxOp_+\BoxOp_-\BoxOp_+ p$ are true at $w$ in
              $\mathcal{M}_1$;
        \item $\DmdOp_-\DmdOp_+ \neg p$ and $\BoxOp_+ p$ are false at $w$ in $\mathcal{M}_2$.
      \end{itemize}
      Therefore,
      \begin{itemize}
        \item there are $v_- \in V_-$ and $v_+ \in V_+$ such that $w \mathrel R v_- \mathrel R v_+$
              and $\mathcal{M}_1, v_+ \nvDash p$;
        \item there is $u_+ \in V_+$ such that $w \mathrel R u_+$ and
              $\mathcal{M}_2, u_+ \nvDash p$.
      \end{itemize}
      Since $v_-, u_+ \in R\{w\}$ and $\mathcal{F} \vDash \mathrm{K4}.3$, one of the following cases
      holds:
      \begin{itemize}
        \item $v_- = u_+$. This is impossible, since $V_- \cap V_+ = \emptyset$;
        \item $v_- \mathrel R u_+$. This is impossible, since
              $\mathcal{M}_2, w \nvDash \DmdOp_-\DmdOp_+ \neg p$;
        \item $u_+ \mathrel R v_-$. This is impossible, since
              $\mathcal{M}_1, w \vDash \BoxOp_+\BoxOp_-\BoxOp_+ p$.
      \end{itemize}
    \end{proof}
    Consider the Kripke frame $\mathcal{F}_0 \defeq (W_0, R_0)$, where
    \begin{gather*}
      W_0 \defeq \{0, 1, 2, 3, 4\},\quad R_0 \defeq \{(i, j) \mid 4 \geq i \geq j \geq 0\} \cup
      \{(0, 1)\}.
    \end{gather*}
    Notice that $\mathcal{F}_0$ is an $\mathrm{S4}.3$-frame.
    \begin{lemma}
      \label{l:lin-2} There is no $p$-positive formula $\alpha(p, r, s)$ such that
      $\mathcal{F}_0 \vDash \varphi \leftrightarrow \alpha$.
    \end{lemma}
    \begin{proof}
      Let $\mathcal{M}_1 \defeq (\mathcal{F}_0, \vartheta_1)$,
      $\mathcal{M}_2 \defeq (\mathcal{F}_0, \vartheta_2)$, where valuations are defined as follows:
      $\vartheta_1(p) = \vartheta_2(p) \defeq \{0\}$,
      $\vartheta_1(r) = \vartheta_2(r) \defeq \{0, 2\}$, $\vartheta_1(s) \defeq \{4\}$, and
      $\vartheta_2(s) \defeq \{3\}$. It is easy to see that $\mathcal{M}_1, 4 \vDash \varphi$ and
      $\mathcal{M}_2, 3 \nvDash \varphi$. At the same time, we will show that
      \begin{gather*}
        Z \defeq \{(0, 0), (1, 1), (2, 2), (2, 0), (3, 1), (4, 3)\}
      \end{gather*}
      is a $p$-directed bisimulation between $\mathcal{M}_1$ and $\mathcal{M}_2$, whence $\varphi$
      can not be equivalent to a $p$-positive formula by \cref{bisim-main}. The condition
      $({\rm lit})$ clearly holds.

      $({\rm zig}_\mathrm{K})$. Let $w_1, v_1, w_2 \in W_0$ be such that $w_1 \mathrel R_0 v_1$ and
      $w_1 \mathrel Z w_2$. We put
      \begin{gather*}
        v_2 \defeq
        \begin{cases}
          0 &\text{if } v_1 \in \{0, 2\},\\
          1 &\text{if } v_1 \in \{1, 3\},\\
          3 &\text{if } v_1 = 4.
        \end{cases}
      \end{gather*}
      It is easy to see that $v_1 \mathrel Z v_2$. If $v_2 \in \{0, 1\}$, then clearly
      $w_2 \mathrel R_0 v_2$. Otherwise, $v_1 = w_1 = 4$, whence $w_2 = 3 \mathrel R_0 v_2$.

      $({\rm zag}_\mathrm{K})$. Let $w_1, w_2, v_2 \in W_0$ be such that $w_1 \mathrel Z w_2$ and
      $w_2 \mathrel R_0 v_2$. We put
      \begin{gather*}
        v_1 \defeq
        \begin{cases}
          v_2 &\text{if } v_2 \in \{0, 1, 2\},\\
          4 &\text{if } v_2 = 3.
        \end{cases}
      \end{gather*}
      It is easy to see that $v_1 \mathrel Z v_2$. If $v_1 \in \{0, 1\}$, then clearly
      $w_1 \mathrel R_0 v_1$. Otherwise, either $v_1 = v_2 = 2$, $w_2 \in \{2, 3\}$, and
      $w_1 \in \{2, 4\}$ or $w_2 = v_2 = 3$ and $w_1 = v_1 = 4$. In both cases,
      $w_1 \mathrel R_0 v_1$.
    \end{proof}
    \begin{proposition}
      \label{ex:lin} ${\rm LPP}(1,2)$ does not hold in all logics between $\mathrm{K4}.3$ and
      $\mathrm{Log}\, \mathcal{F}_0$.
    \end{proposition}
    \begin{proof}
      Follows from \cref{l:lin-1,l:lin-2}.
    \end{proof}
    Notice that a finite frame $\mathcal{F}_0$ was sufficient to prove that $\mathrm{K4}.3$ and
    $\mathrm{S4}.3$ lack LPP. This is not accidental: arguing similar to \cite[Theorem 3.2]{KWZ23},
    one can show that all canonical subframe extensions of $\mathrm{K4}.3$ have the following
    property: if $\varphi$ entails $\psi$ under $\tau$-bisimulations between finite
    $\Lambda$-models, then $\varphi$ entails $\psi$ under $\trianglelefteq_\tau$ on
    $\mathcal{M}_\Lambda$. Therefore, by~\cref{pos-char}, $\varphi$ is $\Lambda$-equivalent to a
    $\vec p$-positive formula iff it is preserved under $\vec p$-directed bisimulations between
    finite $\Lambda$-models. The situation is quite different for the non-canonical extensions of
    $\mathrm{K4}.3$ such as $\mathrm{GL}.3$ and $\mathrm{Grz}.3$. Examples 3.5 and 3.6
    from~\cite{KWZ23} demonstrate that preservation under bisimulations between finite models is not
    sufficient for preservation between arbitrary bisimulations. Moreover, one can prove that
    counterexamples for LPP in $\Lambda \in \{\mathrm{GL}.3, \mathrm{Grz}.3\}$ can not be found
    using finite models: if $\varphi$ is $\vec p$-monotone in $\Lambda$, then it is preserved under
    $\vec p$-directed bisimulations between finite $\Lambda$-models.
  \subsection{Monotonicity vs positivity in tabular extensions of S4}
    A normal logic $\Lambda$ is \emph{tabular} if $\Lambda = \mathrm{Log}\, \mathcal{F}$ for a
    finite Kripke frame $\mathcal{F}$. In this section, we consider particular tabular extensions of
    $\mathrm{S4}$, namely, the logics of the following frames:
    \begin{gather*}
      \mathcal{C}_k \defeq (\underbar{k}, \underbar{k} \times \underbar{k}) \quad\text{and}\quad
      \mathcal{D}_k \defeq (\underbar{k+1}, \underbar{k} \times \underbar{k} \cup \{k\} \times
      \underbar{k+1}),
    \end{gather*}
    where $k \geq 1$, $\underbar{k} \defeq \{i \mid 0 \leq i < k\}$. The frames $\mathcal{C}_k$ are
    also known as \emph{clusters}. It is known that~\cite{GabMaks05,Kur24}
    \begin{itemize}
      \item $\mathrm{Log}\, \mathcal{C}_1 = \mathrm{Triv}$,
            $\mathrm{Log}\, \mathcal{D}_1 = {\rm GW}.2$, and $\mathrm{Log}\, \mathcal{D}_2$ have
            LIP;
      \item $\mathrm{Log}\, \mathcal{C}_2$ has CIP, but does not have LIP;
      \item $\mathrm{Log}\, \mathcal{C}_k$ and $\mathrm{Log}\, \mathcal{D}_k$, $3 \leq k < \omega$
            do not have CIP.
    \end{itemize}
    $\mathrm{Log}\, \mathcal{C}_\omega = \mathrm{S5}$ and
    $\mathrm{Log}\, \mathcal{D}_\omega = \mathrm{S4}.4$ also have LIP, though they are not tabular.
    We will show that
    \begin{itemize}
      \item $\mathrm{Log}\, \mathcal{C}_k$ for $2 \leq k < \omega$ do not have LPP;
      \item $\mathrm{Log}\, \mathcal{D}_k$ for $1 \leq k < \omega$ have LPP.
    \end{itemize}
    LPP for $\mathrm{Triv}$, $\mathrm{S5}$, and $\mathrm{S4}.4$ follows from \cref{lynd}.
    \subsubsection{Logics lacking LPP}
      \begin{proposition}
        \label{ex:cn-par} Let $m, k < \omega$ be such that $2 \leq k \leq 2^m + 1$. Then
        $\mathrm{Log}\, \mathcal{C}_k$ does not have ${\rm LPP}(1,m)$.
      \end{proposition}
      \begin{proof}
        Let $\alpha_i(\bar r), i = 1, \dots, k-1$ be distinct formulas of the form
        \begin{gather*}
          l_0 \wedge \dots \wedge l_{m-1}, \quad\text{where } l_j \in \{r_j\}^\pm
        \end{gather*}
        (if $m = 0$, then $\alpha_1 = \top$). Consider the formula
        $\varphi \defeq \BoxOp p \vee \neg p \wedge \DmdOp\bigwedge_{i=1}^{k-1}(p \wedge \alpha_i)$.

        Let us show that $\varphi$ is $p$-monotone in $\mathcal{C}_k$. Suppose that $\varphi$ is
        true at $w < k$ in a model $\mathcal{M} = (\mathcal{C}_k, \vartheta)$. We will show that
        $\vartheta(p) \supseteq \underbar{k} \setminus \{w\}$. If $\mathcal{M}, w \vDash \BoxOp p$,
        this trivially holds. Otherwise, $w \notin \vartheta(p)$ and there are worlds
        $v_1, \dots, v_{k-1} \in \vartheta(p)$ such that $\mathcal{M}, v_i \vDash \alpha_i$. Since
        $\alpha_i$ are pairwise inconsistent, all $v_i$ are distinct. Also, since
        $w \in \vartheta(p)$ and $v_i \notin \vartheta(p)$, $w \neq v_i$ for all
        $i = 1, \dots, k-1$, whence $\underbar{k} = \{w, v_1, \dots, v_{k-1}\}$ and
        $\vartheta(p) = \underbar{k} \setminus \{w\}$. Now suppose that
        $\vartheta' \geq_p \vartheta$. Then $\vartheta'(p) = \vartheta(p)$ or
        $\vartheta'(p) = \underbar{k}$. Clearly, in both cases
        $(\mathcal{C}_k, \vartheta') \vDash \varphi$.

        It remains to show that $\varphi$ is not equivalent to any $p$-positive formula in
        $\mathcal{C}_k$. Let $\mathcal{M}_1 = (\mathcal{C}_k, \vartheta_1)$ and
        $\mathcal{M}_2 = (\mathcal{C}_k, \vartheta_2)$ be such that
        \begin{itemize}
          \item $\vartheta_1(p) = \underbar{k} \setminus \{0\}$,
                $\vartheta_2(p) = \underbar{k} \setminus \{1\}$;
          \item $\vartheta_1(\alpha_1) = \vartheta_2(\alpha_1) = \{0, 1\}$
          \item $\vartheta_1(\alpha_i) = \vartheta_2(\alpha_i) = \{i\}$ for $2 \leq i < k$.
        \end{itemize}
        It is easy to see that $\mathcal{M}_1, 0 \vDash \varphi$ and
        $\mathcal{M}_2, 0 \nvDash \varphi$. At the same time,
        \begin{gather*}
          Z \defeq \{(0, 0), (0, 1), (1, 0)\} \cup \{(i, i) \mid 2 \leq i < k\}.
        \end{gather*}
        is a $p$-directed bisimulation between $\mathcal{M}_1$ and $\mathcal{M}_2$. Since $\varphi$
        is not preserved under $Z$, it can not be equivalent to a $p$-positive formula by
        \cref{bisim-main}.
      \end{proof}
      Notice that, for $\mathrm{Log}\, \mathcal{C}_2$ we have the following simple counterexample:
      the formula $\BoxOp p \vee \neg p \wedge \DmdOp p$ is monotone in $\mathcal{C}_2$, but is not
      equivalent in it to any positive formula. For $\mathcal{C}_k$ with $k \geq 3$, we used
      parameters, and more parameters were used for larger clusters. This is in fact unavoidable:
      \begin{proposition}
        \label{cn-par-2} Suppose that $k \geq 2^{2n + m}$. Then $\mathrm{Log}\, \mathcal{C}_k$ has
        ${\rm LPP}(n,m)$.
      \end{proposition}
      \begin{proof}
        Let $\vec p = (p_i)_{i<n}$, $\bar r = (r_j)_{j < m}$, $\varphi(\vec p, \bar r)$ be a
        $\vec p$-monotone in $\mathcal{C}_k$ formula. Then the formula $\psi(\vec p, \vec q, \bar r)
        \defeq \varphi(\vec p, \bar r) \to \varphi(\vec p \vee \vec q, \bar r)$ with $2n+m$
        variables is derivable in $\mathrm{Log}\,(\mathcal{C}_k)$. It is well-known that
        $s$-variable fragments of $\mathrm{S5}$ and $\mathrm{Log}\,(\mathcal{C}_{2^s})$ coincide.
        Therefore, $\mathrm{S5} \vdash \psi$, that is, $\varphi$ is $\vec p$-monotone in
        $\mathrm{S5}$. Since $\mathrm{S5}$ has LPP, there is a $\vec p$-positive formula $\psi$ such
        that $\varphi \sim_\mathrm{S5} \psi$. Clearly,
        $\mathcal{C}_k \vDash \varphi \leftrightarrow \psi$.
      \end{proof}
      \begin{corollary}
        For all $n, m < \omega$, there is a logic $\Lambda$ with ${\rm LPP}(n,m)$ but without
        ${\rm LPP}$.
      \end{corollary}
      At the same time, we do not know examples of logics with ${\rm LPP}(1,\omega)$ or
      ${\rm LPP}(\omega,0)$, which do not have LPP.
    \subsubsection{Preservation theorem for tabular normal logics}
      To prove that LPP holds in $\mathrm{Log}\, \mathcal{D}_k$, we need to refine the preservation
      theorem for positive formulas in tabular normal logics.

      Let $\mathcal{F} = (W, R)$ be a Kripke frame, $\mathcal{M} = (\mathcal{F}, \vartheta)$ be a
      model on it. For $V \subseteq W$, we denote by $\mathcal{F}|_V$ the frame $(V, R|_V)$, where
      $R|_V \defeq R \cap (V \times V)$, and by $\mathcal{M}|_V$ the model
      $(\mathcal{F}|_V, \vartheta|_V)$, where $\vartheta|_V$ maps each variable $p$ to
      $\vartheta(p) \cap V$. Let us also denote by $R^*$ the reflexive transitive closure of $R$.
      $V \subseteq W$ is a \emph{cone} if $V = R^*\{v\}$ for some $v \in V$.
      \begin{definition}
        Let $\mathcal{F}_1 = (W_1, R_1)$ and $\mathcal{F}_2 = (W_2, R_2)$ be Kripke frames,
        $V \subseteq W_1$ be a cone. An onto morphism
        $f : \mathcal{F}_1|_{V} \twoheadrightarrow \mathcal{F}_2$, is called a \emph{reduction} from
        $\mathcal{F}_1$ to $\mathcal{F}_2$.
      \end{definition}
      The following results are known (see~\cite[Section 1.14]{GSS09}):
      \begin{enumerate}
        \item Every tabular logic is canonical. \label{tab-1}
        \item If $\Lambda$ is tabular and $\mathcal{F}$ is a $\Lambda$-frame, then every cone in
              $\mathcal{F}$ is finite. \label{tab-2}
        \item For finite frames,
              $\mathrm{Log}\, \mathcal{F}_1 \subseteq \mathrm{Log}\, \mathcal{F}_2$ iff
              $\mathcal{F}_1$ is reducible to $\mathcal{F}_2$. \label{tab-3}
      \end{enumerate}
      \begin{lemma}
        \label{tab-red} Let $\Lambda$ be the logic of a finite Kripke frame $\mathcal{F}$,
        $w \in W_\Lambda$, $V \defeq R_\Lambda^*\{w\}$. Then $\mathcal{F}$ is reducible to
        $\mathcal{F}_\Lambda|_V$.
      \end{lemma}
      \begin{proof}
        By \cref{tab-1}, $\mathcal{F}_\Lambda$ is a $\Lambda$-frame. By \cref{tab-2}, $V$ is finite.
        Since $\mathrm{Log}\, \mathcal{F} = \Lambda = \mathrm{Log}\, \mathcal{F}_\Lambda \subseteq
        \mathrm{Log}\, \mathcal{F}_\Lambda|_V$, $\mathcal{F}$ is reducible to
        $\mathcal{F}_\Lambda|_V$ by \cref{tab-3}.
      \end{proof}
      \begin{proposition}
        Let $\mathcal{F}$ be a finite Kripke frame, $\Lambda = \mathrm{Log}\, \mathcal{F}$,
        $\tau \subseteq \mathcal{V}^\pm$. Then the following conditions on formulas $\varphi$ and
        $\psi$ are equivalent:
        \begin{enumerate}
          \item there is a $\tau$-interpolant for $\varphi, \psi \in \mathrm{Fm}$ in $\Lambda$;
                \label{it-1}
          \item $\varphi$ entails $\psi$ under all $\tau$-bisimulations between models on
                $\mathcal{F}$. \label{it-2}
        \end{enumerate}
      \end{proposition}
      \begin{proof}
        Since models on $\mathcal{F}$ are $\Lambda$-models, by \cref{lam-char}, \cref{it-1}
        $\Rightarrow$ \cref{it-2}. For the converse, suppose that $\varphi$ and $\psi$ do not have a
        $\tau$-interpolant in $\Lambda$. Then, by \cref{lam-char}, $\varphi$ does not entail $\psi$
        under $\trianglelefteq_\tau$ on $\mathcal{M}_\Lambda$, that is,
        $w_1 \trianglelefteq_\tau w_2$, $\mathcal{M}_\Lambda, w_1 \vDash \varphi$, and
        $\mathcal{M}_\Lambda, w_2 \nvDash \psi$ for some $w_1, w_2 \in W_\Lambda$. By
        \cref{tab-red}, there are reductions $f_1$ and $f_2$ from $\mathcal{F}$ to
        $R_\Lambda^*\{w_1\}$ and $R_\Lambda^*\{w_2\}$. Consider the models
        $\mathcal{M}_1 \defeq (\mathcal{F}, \vartheta_1)$ and
        $\mathcal{M}_2 \defeq (\mathcal{F}, \vartheta_2)$, where
        $\vartheta_1 \defeq f_1^{-1}\vartheta_\Lambda$ and
        $\vartheta_2 \defeq f_2^{-1}\vartheta_\Lambda$. Then
        $f_1 : \mathcal{M}_1 \to \mathcal{M}_\Lambda$ and
        $f_2 : \mathcal{M}_2 \to \mathcal{M}_\Lambda$ are morphisms. Choose some worlds
        $w_1' \in f_1^{-1}\{w_1\}$ and $w_2' \in f_2^{-1}\{w_2\}$. Then
        $\mathcal{M}_1, w_1 \vDash \varphi$ and $\mathcal{M}_2, w_2 \nvDash \psi$. Moreover,
        $Z \defeq f_2^{-1} {\trianglelefteq_\tau} f_1$ is a composition of $\tau$-bisimulations,
        whence it is also a $\tau$-bisimulation. Clearly, $w_1' \mathrel Z w_2'$.
      \end{proof}
      \begin{corollary}
        \label{tab-pos-char} Let $\mathcal{F}$ be a finite frame,
        $\Lambda = \mathrm{Log}\, \mathcal{F}$, $\varphi \in \mathrm{Fm}$. Then $\varphi$ is
        equivalent to a $\vec p$-positive formula iff it is preserved under all $\vec p$-directed
        bisimulations between models on $\mathcal{F}$.
      \end{corollary}
    \subsubsection{Logics having LPP}
      \begin{lemma}
        \label{ppqq} Suppose that $Z \subseteq W_1 \times W_2$ is a full relation. Then there is a
        zigzag-free full relation $Z_0 \subseteq Z$.
      \end{lemma}
      \begin{proof}
        We proceed by complete induction on $|W_1| + |W_2|$. If $W_1$ or $W_2$ is empty, then
        another set is also empty and the statement is trivial. Now, assume that both sets are
        non-empty.

        Suppose that there is $w_1 \in W_1$ such that $|Z\{w_1\}| = 1$. Let $w_2 \in W_2$ be such
        that $w_1 \mathrel Z w_2$ and consider the set
        \begin{gather*}
          V \defeq \{v_1 \in W_1 \mid Z\{v_1\} = \{w_2\} \}.
        \end{gather*}
        We put $W_1' \defeq W_1 \setminus V$, $W_2' \defeq W_2 \setminus \{w_2\}$. It is easy to see
        that $Z' \defeq Z \cap (W_1' \times W_2')$ is a full relation between $W_1'$ and $W_2'$. By
        the induction hypothesis, there is a zigzag-free full relation $Z'_0 \subseteq Z'$ between
        $W_1'$ and $W_2'$. Then, $Z_0 \defeq Z'_0 \cup V \times \{w_2\} \subseteq Z$ is a
        zigzag-free full relation between $W_1$ and $W_2$.

        If $|Z^{-1}\{w_2\}| = 1$ for some $w_2 \in W_2$, then the statement of the lemma also holds
        by the symmetric argument.

        Now suppose that $|Z\{w_1\}| > 1$ and $|Z^{-1}\{w_2\}| > 1$ for all $w_1 \in W_1$ and
        $w_2 \in W_2$. Let us fix arbitrary pair $(w_1, w_2) \in Z$ and put
        $W_1' \defeq W_1 \setminus \{w_1\}$ and $W_2' \defeq W_2 \setminus \{w_2\}$. It is easy to
        see that $Z' \defeq Z \cap (W_1' \times W_2')$ is a full relation between $W_1'$ and $W_2'$.
        By the induction hypothesis, there is a zigzag-free full relation $Z'_0 \subseteq Z'$
        between $W_1'$ and $W_2'$. Then, $Z_0 \defeq Z'_0 \cup \{(w_1, w_2)\} \subseteq Z$ is a
        zigzag-free full relation between $W_1$ and $W_2$.
      \end{proof}
      \begin{lemma}
        \label{qqq} Let $Z$ be a non-empty relation on $\underbar{k+1}$. Then $Z$ is a bisimulation
        on $\mathcal{D}_k$ iff $Z|_{\underbar{k}} \defeq Z \cap (\underbar{k} \times \underbar{k})$
        is a full relation on $\underbar{k}$.
      \end{lemma}
      \begin{proof}
        Let $R_k$ be the accessibility relation in $\mathcal{D}_k$.

        ($\Rightarrow$). By symmetry, it is sufficient to show that there is $j < k$ such that
        $0 \mathrel Z j$. Let us fix some $(w_1, w_2) \in Z$. Since $w_2 \mathrel R_k 0$,
        $v_1 \mathrel Z 0$ for some $v_1 \in R_k\{w_1\}$. Since $v_1 \mathrel R_k 0$ and
        $R_k\{0\} = \underbar{k}$, $0 \mathrel Z j$ for some $j < k$.

        ($\Leftarrow$). By symmetry, it is sufficient to check $({\rm zig}_\mathrm{K})$. Let
        $w_1, w_2, v_1 < k+1$ be such that $w_1 \mathrel R_k v_1$ and $w_1 \mathrel Z w_2$. If
        $v_1 < k$, then there is $v_2 < k$ such that $v_1 \mathrel Z v_2$. Clearly,
        $w_2 \mathrel R_k v_2$. Otherwise, $w_1 = v_1 = k$, whence $v_1 \mathrel Z v_2$ and
        $w_2 \mathrel R_k v_2$ for $v_2 \defeq w_2$.
      \end{proof}
      \begin{proposition}
        \label{ex:dk} For all $k \in \omega \setminus \{0\}$, $\mathrm{Log}\, \mathcal{D}_k$ has
        LPP.
      \end{proposition}
      \begin{proof}
        By \cref{tab-pos-char}, it is sufficient to show that every $\vec p$-monotone in
        $\mathcal{D}_k$ formula $\varphi$ is preserved under $\vec p$-directed bisimulations between
        models on $\mathcal{D}_k$. Let $\mathcal{M}_1 = (\mathcal{D}_k, \vartheta_1)$ and
        $\mathcal{M}_2 = (\mathcal{D}_k, \vartheta_2)$ be models, $Z$ be a $\vec p$-directed
        bisimulation between $\mathcal{M}_1$ and $\mathcal{M}_2$, $(w_1, w_2) \in Z$ be such that
        $\mathcal{M}_1, w_1 \vDash \varphi$.

        Suppose that $w_1 = w_2 = k$. By \cref{qqq}, $Z|_{\underbar{k}}$ is a full relation on
        $\underbar{k}$. By \cref{ppqq}, there is a zigzag-free full relation
        $Z_0 \subseteq Z|_{\underbar{k}}$ on $\underbar{k}$. Then, by \cref{qqq}
        $Z' \defeq Z_0 \cup \{(n, n)\}$ is a bisimulation on $\mathcal{D}_k$, whence it is clearly a
        zigzag-free $\vec p$-directed bisimulation between $\mathcal{M}_1$ and $\mathcal{M}_2$. By
        \cref{mon-char}, $\mathcal{M}_2, k \vDash \varphi$.

        Now, suppose that $w_1 = k$ and $w_2 < k$. Consider the model
        $\mathcal{M}' \defeq (\mathcal{D}_k, \vartheta')$ such that
        $\mathcal{M}'|_{\underbar k} = \mathcal{M}_2|_{\underbar k}$ and, for all
        $q \in \mathcal{V}$, $\mathcal{M}', k \vDash q \Leftrightarrow \mathcal{M}_2, w_2 \vDash k$.
        Let $Z' \defeq Z|_{\underbar k} \cup \{(k, k)\}$. By \cref{qqq}, $Z'$ is a bisimulation on
        $\mathcal{D}_k$, whence it is clearly a $\vec p$-directed bisimulation between
        $\mathcal{M}_1$ and $\mathcal{M}'$. By the previous case, $\mathcal{M}', k \vDash \varphi$.
        At the same time, it is easy to see that $\{(k, w_2)\} \cup \{(i, i) \mid i < k\}$ is a
        morphism from $\mathcal{M}'$ into $\mathcal{M}_2$, whence
        $\mathcal{M}_2, w_2 \vDash \varphi$.

        Finally, suppose that $w_1 < k$. Consider the model
        $\mathcal{M}' \defeq (\mathcal{D}_k, \vartheta')$ such that
        $\mathcal{M}'|_{\underbar k} = \mathcal{M}_1|_{\underbar k}$ and, for all
        $q \in \mathcal{V}$, $\mathcal{M}', k \vDash q \Leftrightarrow \mathcal{M}_1, w_1 \vDash k$.
        Clearly, $\{(w_1, k)\} \cup \{(i, i) \mid i < k\}$ is a bisimulation between $\mathcal{M}_1$
        and $\mathcal{M}'$, whence $\mathcal{M}', k \vDash \varphi$. Let
        $Z' \defeq Z|_{\underbar k} \cup \{(k, w_2)\}$. By \cref{qqq}, $Z'$ is a bisimulation on
        $\mathcal{D}_k$, whence it is a $\vec p$-directed bisimulation between $\mathcal{M}'$ and
        $\mathcal{M}_2$. By the previous cases, $\mathcal{M}', w_2 \vDash \varphi$.
      \end{proof}

  \section{Interpolation and LPP in non-normal logic}
  \label{s:non-norm} In this section, we are going to investigate the Lyndon positivity property in
  non-normal logics. To establish LPP we will use a construction known as bisimulation or zigzag
  product of frames. In \cite[Section 5.2]{Marx95}, it was shown that every canonical normal logic
  which is preserved under bisimulation products of Kripke frames has CIP. In \cite[Section
  9.2]{Han03}, bisimulation products of neighborhood frames were defined and it was noted that, for
  $\sigma$- and $\pi$-canonical monotone logics which are preserved under bisimulation products, CIP
  can be proven in a similar way. In this section, we refine these results in two directions.
  Firstly, we show that all such logics have not only CIP, but also LIP and LPP (recall that it is
  unknown whether $LIP \Rightarrow {\rm LPP}$ in non-normal logics). Secondly, we provide an
  infinite family of non-normal logics which satisfy these conditions (that is, are
  $\sigma$-canonical and are preserved under bisimulation products, whence have both LIP and LPP),
  namely all logics axiomatizable over $\mathrm{EM}$ by means of closed formulas and formulas of the
  form $\alpha(p) \to \DmdOp p$, where $\alpha$ is positive.
  \subsection{Bisimulation products of neighborhood models}
    \begin{definition}
      Let $Z \subseteq W_1 \times W_2$ be a full bisimulation between neighborhood frames
      $\mathfrak{F}_1 = (W_1, \DmdOp_1)$ and $\mathfrak{F}_2 = (W_2, \DmdOp_2)$. A frame
      $\mathfrak{F} = (Z, \DmdOp)$ is a \emph{bisimulation product} of $\mathfrak{F}_1$ and
      $\mathfrak{F}_2$ if, for $k = 1, 2$, $\pi_k : (w_1, w_2) \mapsto w_k$ are morphisms from
      $\mathfrak{F}$ to $\mathfrak{F}_k$, that is,
      \begin{gather}
        \DmdOp\pi_k^{-1}X_k = \pi_k^{-1}(\DmdOp_kX_k) \quad\text{for } k = 1, 2,\;X_k \subseteq W_k.
        \label{eq:bisim-prod}
      \end{gather}
    \end{definition}
    Notice that the following equalities trivially hold:
    \begin{gather*}
      \pi_1\pi_1^{-1} = 1_{W_1},\quad \pi_2\pi_2^{-1} = 1_{W_2},\quad \pi_2\pi_1^{-1} = Z,\quad
      \pi_1\pi_2^{-1} = Z^{-1}.
    \end{gather*}
    Similarly to the case of canonical model, we can state the condition~\cref{eq:bisim-prod} in
    terms of operator extensions. For $k = 1, 2$, let
    $D_k \defeq \{ \pi_k^{-1} X_k \mid X_k \subseteq W_k\}$. It is easy to see that
    $(\mathfrak{P}(W_1 \times W_2)|_{D_k}, \pi_k^{-1}\DmdOp_k\pi_k)$ is a modal algebra which is
    isomorphic to $\mathfrak{F}_k^*$. Then $\DmdOp$ satisfies~\cref{eq:bisim-prod} iff it is a
    common extension of operators $\pi_1^{-1}\DmdOp_1\pi_1$ defined on $D_1$ and
    $\pi_2^{-1}\DmdOp_2\pi_2$ defined on $D_2$.
    \begin{definition}
      A logic $\Lambda$ is \emph{preserved under (some) bisimulation products} if, for every full
      bisimulation $Z$ between $\Lambda$-frames $\mathfrak{F}_1$ and $\mathfrak{F}_2$, there is an
      operator $\DmdOp$ on $\mathcal{P}(Z)$ such that $\mathfrak{F} = (Z, \DmdOp)$ is a bisimulation
      product of $\mathfrak{F}_1$ and $\mathfrak{F}_2$ and $\mathfrak{F} \vDash \Lambda$.
    \end{definition}

    Consider the following operator on $\mathcal{P}(Z)$:
    \begin{gather*}
      \DmdOp_\mathrm{max} X \defeq \pi_1^{-1}\DmdOp_1\pi_1X \cap \pi_2^{-1}\DmdOp_2\pi_2X.
    \end{gather*}
    \vspace{-3.5ex}
    \begin{lemma}
      $\DmdOp_\mathrm{max}$ is the greatest monotone operator satisfying~\cref{eq:bisim-prod}.
    \end{lemma}
    \begin{proof}
      Notice that, for $X_1 \subseteq W_1$,
      \begin{gather*}
        \DmdOp_\mathrm{max} \pi_1^{-1}X_1 = \pi_1^{-1}\DmdOp_1\pi_1 (\pi_1^{-1}X_1) \cap
        \pi_2^{-1}\DmdOp_2\pi_2 (\pi_1^{-1}X_1) = \pi_1^{-1}\DmdOp_1 X_1 \cap \pi_2^{-1}\DmdOp_2
        ZX_1
      \end{gather*}
      and $\pi_2^{-1}\DmdOp_2Z X_1 \supseteq \pi_2^{-1}Z\DmdOp_1 X_1 = \pi_2^{-1}\pi_2\pi_1^{-1}
      \DmdOp_1X_1 \supseteq \pi_1^{-1}\DmdOp_1X_1$. Therefore, $\DmdOp_\mathrm{max}$ satisfies
      \cref{eq:bisim-prod} for $k = 1$. By symmetry, it also satisfies \cref{eq:bisim-prod} for
      $k = 2$.

      Suppose that some monotone operator $\DmdOp$ satisfies~\cref{eq:bisim-prod}. Since
      $\pi_k^{-1}\pi_k \geq 1_Z$,
      $\DmdOp X \subseteq \DmdOp\pi_k^{-1}\pi_k X = \pi_k^{-1}\DmdOp_k\pi_k X$ for $X \subseteq Z$.
      Thus, $\DmdOp X \subseteq \DmdOp_\mathrm{max} X$.
    \end{proof}
    \begin{definition}
      A logic $\Lambda$ is \emph{preserved under maximal bisimulation products} if, for every full
      bisimulation $Z$ between $\Lambda$-frames $\mathfrak{F}_1$ and $\mathfrak{F}_2$,
      $(Z, \DmdOp_\mathrm{max})$ is a $\Lambda$-frame. A formula $\varphi$ is \emph{preserved under
      maximal bisimulation products} if $\mathrm{EM} + \varphi$ is preserved under maximal
      bisimulation products.
    \end{definition}
    Clearly, if all formulas $\varphi \in \Gamma$ are preserved under maximal bisimulation products,
    then the logic $\mathrm{EM} + \Gamma$ is preserved under maximal bisimulation products. It is
    easy to see that, if $\varkappa$ is a closed formula and $\mathfrak{F} = (Z, \DmdOp)$ is a
    bisimulation product of $\mathfrak{F}_1$ and $\mathfrak{F}_2$, then $\mathfrak{F}_1 \vDash
    \varkappa \Leftrightarrow \mathfrak{F} \vDash \varkappa \Leftrightarrow \mathfrak{F}_2 \vDash
    \varkappa$. In this sense closed formulas are preserved under all (in particular, maximal)
    bisimulation products. In~\cite[Proposition 9.9]{Han03}, it was proven that ${\rm A}4$,
    ${\rm AT}$, and ${\rm AP}$ are preserved under maximal bisimulation products. We are going to
    show that the same holds for all formulas of the form $\alpha(p) \to \DmdOp p$, where $\alpha$
    is positive.
    \begin{lemma}
      \label{l:pos-prod} Let $\mathfrak{F} = (Z, \DmdOp_\mathrm{max})$ be a maximal bisimulation
      product of $\mathfrak{F}_1 = (W_1, \DmdOp_1)$ and $\mathfrak{F}_2 = (W_2, \DmdOp_2)$,
      $\alpha(p)$ be a positive formula. Then,
      \begin{gather*}
        \alpha_\mathfrak{F} X \subseteq \pi_1^{-1}\alpha_{\mathfrak{F}_1}\pi_1 X \cap
        \pi_2^{-1}\alpha_{\mathfrak{F}_2}\pi_2 X
        \quad\text{for all } X \subseteq Z.
      \end{gather*}
    \end{lemma}
    \begin{proof}
      We proceed by induction on construction of $\alpha$. By symmetry, it is sufficient to prove
      that $\alpha_\mathfrak{F} X \subseteq \pi_1^{-1}\alpha_{\mathfrak{F}_1}\pi_1 X$.

      For $\alpha = \bot$, $\emptyset = \pi_1^{-1} \emptyset$ is trivial.

      For $\alpha = \top$, $Z = \pi_1^{-1}W_1$, since $Z$ is full.

      For $\alpha = p$, $X \subseteq \pi_1^{-1}\pi_1 X$, since $\pi_1^{-1}\pi_1 \geq 1_Z$.

      For $\alpha = \beta \wedge \gamma$, by the induction hypothesis,
      \begin{gather*}
        \alpha_\mathfrak{F} X = \beta_\mathfrak{F} X \cap \gamma_\mathfrak{F} X \subseteq
        \pi_1^{-1}\beta_{\mathfrak{F}_1}\pi_1 X \cap \pi_1^{-1}\gamma_{\mathfrak{F}_1}\pi_1 X
      \end{gather*}
      Since $\pi_1^{-1}$ is a boolean algebra homomorphism, the last term equals
      \begin{gather*}
        \pi_1^{-1}(\beta_{\mathfrak{F}_1}\pi_1 X \cap \gamma_{\mathfrak{F}_1}\pi_1 X) =
        \pi_1^{-1}\alpha_{\mathfrak{F}_1}\pi_1 X.
      \end{gather*}
      For $\alpha = \beta \vee \gamma$ the argument is similar.

      For $\alpha = \DmdOp\beta$, by the induction hypothesis,
      \begin{gather*}
        \alpha_\mathfrak{F} X = \DmdOp_\mathrm{max}(\beta_\mathfrak{F} X) \subseteq
        \pi_1^{-1}\DmdOp_1\pi_1 (\pi_1^{-1}\beta_{\mathfrak{F}_1}\pi_1 X).
      \end{gather*}
      Since $\pi_1\pi_1^{-1} = 1_{W_1}$, the last expression equals
      $\pi_1^{-1}\alpha_{\mathfrak{F}_1}\pi_1 X$.

      For $\alpha = \BoxOp\beta$, by the induction hypothesis
      \begin{align*}
        \alpha_\mathfrak{F} X = {-}\DmdOp_\mathrm{max}{-}(\beta_\mathfrak{F} X) &\subseteq
        -\bigcap_{k=1,2}\bigl(\pi_k^{-1}\DmdOp_k\pi_k{-} (\pi_1^{-1}\beta_{\mathfrak{F}_1}\pi_1
        X)\bigr)\\
        &= \bigcup_{k=1,2}\bigl(\pi_k^{-1}\BoxOp_k{-}\pi_k\pi_1^{-1}{-}(\beta_{\mathfrak{F}_1}\pi_1
        X)\bigr).
      \end{align*}
      Consider two sets in the union separately. For $k = 1$, since $\pi_1\pi_1^{-1} = 1_{W_1}$,
      \begin{gather*}
        \pi_1^{-1}\BoxOp_1{-}\pi_1\pi_1^{-1}{-}(\beta_{\mathfrak{F}_1}\pi_1 X) =
        \pi_1^{-1}\BoxOp_1(\beta_{\mathfrak{F}_1}\pi_1 X) = \pi_1^{-1}\alpha_{\mathfrak{F}_1}\pi_1
        X.
      \end{gather*}
      For $k = 2$, since $\pi_1^{-1}\pi_1 \geq 1_Z$,
      \begin{align*}
        \pi_2^{-1}\BoxOp_2{-}\pi_2\pi_1^{-1}{-}(\beta_{\mathfrak{F}_1}\pi_1 X)
        &\subseteq
        \pi_1^{-1}\pi_1\pi_2^{-1}\BoxOp_2{-}\pi_2\pi_1^{-1}{-}(\beta_{\mathfrak{F}_1}\pi_1 X)\\
        &= \pi_1^{-1}Z^{-1}\BoxOp_2{-}Z{-}(\beta_{\mathfrak{F}_1}\pi_1 X)\\
        &\subseteq \pi_1^{-1}\BoxOp_1Z^{-1}{-}Z{-}(\beta_{\mathfrak{F}_1}\pi_1 X)\\
        &\subseteq \pi_1^{-1}\BoxOp_1\beta_{\mathfrak{F}_1}\pi_1 X
      \end{align*}
      since $Z^{-1}{-}Z{-} \leq 1_{W_1}$
    \end{proof}
    \begin{proposition}
      \label{prod-preserv} Formulas of the form $\alpha(p) \to \DmdOp p$, where $\alpha(p)$ is
      positive, are preserved under maximal bisimulation products.
    \end{proposition}
    \begin{proof}
      Suppose that $\alpha(p) \to \DmdOp p$ is valid in $\mathfrak{F}_1$ and $\mathfrak{F}_2$. Then,
      for each $X \subseteq Z$ and $k = 1, 2$,
      $\alpha_{\mathfrak{F}_k}\pi_kX \subseteq \DmdOp_k\pi_kX$. By \cref{l:pos-prod},
      \begin{gather*}
        \alpha_\mathfrak{F} X \subseteq \bigcap_{k=1,2}\pi_k^{-1}\alpha_{\mathfrak{F}_k}\pi_k X
        \subseteq \bigcap_{k=1,2}\pi_k^{-1}\DmdOp_k\pi_kX = \DmdOp_\mathrm{max} X.
      \end{gather*}
      Therefore, $\mathfrak{F} \vDash \alpha \to \DmdOp p$.
    \end{proof}
  \subsection{Interpolation and MPP via bisimulation products}
    \begin{lemma}
      \label{cp-int} Let $\Lambda$ be a monotone $\sigma$- or $\pi$-canonical logic which is
      preserved under bisimulation products, $\tau$ be a set of literals. Then there is a
      $\Lambda$-frame $\mathfrak{F}$ and valuations $\vartheta_1$ and $\vartheta_2$ on it such that
      $\vartheta_1 \leq_\tau \vartheta_2$ and the following holds: if there is no $\tau$-interpolant
      in $\Lambda$ for some $\varphi, \psi \in \mathrm{Fm}$, then
      $\vartheta_1(\varphi) \setminus \vartheta_2(\psi) \neq \emptyset$.
    \end{lemma}
    \begin{proof}
      Let $s = \sigma$ if $\Lambda$ is $\sigma$-canonical and $s = \pi$ otherwise. By
      \cref{can-sim}, $\trianglelefteq_\tau$ is a full $\tau$-bisimulation on
      $\mathfrak{M}^s_\Lambda$. Let $\mathfrak{F} = (\trianglelefteq_\tau, \DmdOp)$ be a
      bisimulation product of $\mathcal{F}^s$ and $\mathcal{F}^s$ such that
      $\mathfrak{F} \vDash \Lambda$. Consider the valuations
      $\vartheta_1 \defeq \pi_1^{-1}\vartheta_\Lambda$,
      $\vartheta_2 \defeq \pi_2^{-1}\vartheta_\Lambda$ on $\mathfrak{F}$. For $l \in \tau$ and
      $v_1 \trianglelefteq_\tau v_2$, if $v_1 \in \vartheta_\Lambda(l)$, then
      $v_2 \in \vartheta_\Lambda(l)$. Therefore, $\vartheta_1 \leq_\tau \vartheta_2$.

      Now, if there is no $\tau$-interpolant for $\varphi$ and $\psi$ in $\Lambda$, then, by
      \cref{lam-char}, there are $w_1 \in \vartheta_\Lambda(\varphi)$ and
      $w_2 \in \vartheta_\Lambda(\neg\psi)$ such that $w_1 \trianglelefteq_\tau w_2$. Clearly,
      $(w_1, w_2) \in \vartheta_1(\varphi) \setminus \vartheta_2(\psi)$ in this case.
    \end{proof}
    \begin{theorem}
      \label{t:prod-mp} Let $\Lambda$ be a monotone $\sigma$- or $\pi$-canonical logic which is
      preserved under bisimulation products. Then $\Lambda$ has LPP.
    \end{theorem}
    \begin{proof}
      Let $\tau \defeq \mathcal{V}^\pm \setminus \neg\vec p$ for some tuple of variables $\vec p$;
      $\mathfrak{F}$, $\vartheta_1$, and $\vartheta_2$ be as in \cref{cp-int}. If
      $\varphi \in \mathrm{Fm}$ is not equivalent to any $\vec p$-positive formula in $\Lambda$,
      then, there is no $\tau$-interpolant for $\varphi$ and $\varphi$ in $\Lambda$, whence
      $\vartheta_1(\varphi) \setminus \vartheta_2(\varphi) \neq \emptyset$. At the same time,
      $\vartheta_1 \leq_\tau \vartheta_2$. Thus, $\varphi$ is not $\vec p$-monotone in $\Lambda$.
    \end{proof}
    \begin{lemma}
      \label{val-inter} Suppose that valuations $\vartheta_1$ and $\vartheta_2$ on
      $\mathfrak{F} = (W, \DmdOp)$ are such that $\vartheta_1 \leq_\tau \vartheta_2$, where
      $\tau = \tau_1 \cap \tau_2$ for some sets of literals $\tau_1$ and $\tau_2$. Then there is a
      valuation $\vartheta$ on $\mathfrak{F}$ such that
      $\vartheta_1 \leq_{\tau_1} \vartheta \leq_{\tau_2} \vartheta_2$.
    \end{lemma}
    \begin{proof}
      Notice that $\vartheta_1 \leq_{\tau_1} \vartheta \leq_{\tau_2} \vartheta_2$ iff the following
      conditions are satisfied for each $p \in \mathcal{V}$:
      \begin{enumerate}
        \item if $p \in \tau_1$, then $\vartheta_1(p) \subseteq \vartheta(p)$; \label{c:1}
        \item if $\neg p \in \tau_1$, then $\vartheta_1(p) \supseteq \vartheta(p)$; \label{c:2}
        \item if $p \in \tau_2$, then $\vartheta(p) \subseteq \vartheta_2(p)$; \label{c:3}
        \item if $\neg p \in \tau_2$, then $\vartheta(p) \supseteq \vartheta_2(p)$. \label{c:4}
      \end{enumerate}
      Let us fix some variable $p$. Notice that
      \begin{itemize}
        \item $\vartheta(p) \defeq \vartheta_1(p)$ satisfies conditions \cref{c:1} and \cref{c:2};
        \item $\vartheta(p) \defeq \vartheta_2(p)$ satisfies conditions \cref{c:3} and \cref{c:4};
        \item $\vartheta(p) \defeq \emptyset$ satisfies conditions \cref{c:2} and \cref{c:3};
        \item $\vartheta(p) \defeq W$ satisfies conditions \cref{c:1} and \cref{c:4}.
      \end{itemize}
      If $p \notin \tau_1 \cap \tau_2$ and $\neg p \notin \tau_1 \cap \tau_2$, then one of
      conditions \cref{c:1} and \cref{c:3} and one of conditions \cref{c:2} and \cref{c:4} are
      vacuously satisfied. Remaining two conditions are satisfied by one of the above valuations.

      Now, suppose that $p \in \tau_1 \cap \tau_2$. Then, since $\vartheta_1 \leq_\tau \vartheta_2$,
      $\vartheta_1(p) \subseteq \vartheta_2(p)$. If $\neg p \notin \tau_1$, then condition
      \cref{c:2} is vacuously satisfied and $\vartheta(p) \defeq \vartheta_2(p)$ satisfies all other
      conditions. If $\neg p \notin \tau_2$, then condition \cref{c:4} is vacuously satisfied and
      $\vartheta(p) \defeq \vartheta_1(p)$ satisfies all other conditions. Finally, if
      $\neg p \in \tau_1 \cap \tau_2$, then $\vartheta_1(p) \supseteq \vartheta_2(p)$ and
      $\vartheta(p) \defeq \vartheta_1(p) = \vartheta_2(p)$ satisfies all conditions.

      The case $\neg p \in \tau_1 \cap \tau_2$ is symmetrical to the previous one.
    \end{proof}
    \begin{theorem}
      \label{t:prod-lip} Let $\Lambda$ be a monotone $\sigma$- or $\pi$-canonical logic which is
      preserved under bisimulation products. Then $\Lambda$ has LIP.
    \end{theorem}
    \begin{proof}
      Let us fix some formulas $\varphi$, $\psi$ and put
      $\tau \defeq \mathrm{lits}(\varphi) \cap \mathrm{lits}(\psi)$. Suppose that there is no Lyndon
      interpolant for $\varphi$ and $\psi$ in $\Lambda$. Then, by \cref{cp-int}, there is a
      $\Lambda$-frame $\mathfrak{F}$ and valuations $\vartheta_1 \leq_\tau \vartheta_2$ on it such
      that $\vartheta_1(\varphi) \setminus \vartheta_2(\psi) \neq \emptyset$. By \cref{val-inter},
      there is a valuation $\vartheta$ on $\mathfrak{F}$ such that
      $\vartheta_1 \leq_{\mathrm{lits}(\varphi)} \vartheta \leq_{\mathrm{lits}(\psi)} \vartheta_2$.
      Then $\vartheta_1(\varphi) \subseteq \vartheta(\varphi)$ and
      $\vartheta(\psi) \subseteq \vartheta_2(\psi)$, whence
      $\vartheta(\varphi) \setminus \vartheta(\psi) \neq \emptyset$. Thus,
      $\Lambda \nvdash \varphi \to \psi$.
    \end{proof}
  \begin{corollary}
    \label{cor:lip-mp} Suppose that $\Lambda$ is axiomatizable over $\mathrm{EM}$ by means of closed
    formulas and formulas of the form $\alpha(p) \to \DmdOp p$, where $\alpha$ is positive. Then
    $\Lambda$ has LIP and LPP.
  \end{corollary}
  \begin{proof}
    Such logics are canonical by \cref{KW-canon} and are preserved under maximal bisimulation
    products by \cref{prod-preserv}. Then LPP follows from \cref{t:prod-mp} and LIP follows from
    \cref{t:prod-lip}.
  \end{proof}
  \begin{corollary}
    Suppose that $\Lambda$ is axiomatizable over $\mathrm{EM}$ by means of closed formulas and
    formulas of the form $\alpha(p) \to \BoxOp p$, where $\alpha$ is positive. Then $\Lambda$ has
    LIP and LPP.
  \end{corollary}
  \begin{proof}
    Notice that for such logics $\Lambda^d$ satisfy the conditions of \cref{cor:lip-mp}, whence
    $\Lambda^d$ has LIP and LPP. Then it is easy to see that $\Lambda$ also has LIP and LPP.
  \end{proof}

  \section{Open questions}
  \begin{itemize}
    \item Is it true that ${\rm LIP} \Rightarrow {\rm LPP}$ in non-normal monotone logics
          (cf.~\cref{lynd})?
    \item Is it true that ${\rm CIP}$ and ${\rm LPP} \Rightarrow {\rm LIP}$ in monotone/normal
          logics?
    \item Is it true that ${\rm LPP}(1,\omega) \Rightarrow {\rm LPP}$ in non-normal monotone logics
          (cf.~\cref{par-norm})?
    \item Is it true that ${\rm LPP}(\omega,0) \Rightarrow {\rm LPP}$ in monotone/normal logics
          without LIP (cf.~\cref{lpp-mono})?
    \item Describe all extensions of $\mathrm{S4}$ with LPP.
    \item Does LPP hold in $\mathrm{GL}.3$ and $\mathrm{Grz}.3$?
    \item Is it true that $\varphi$ is equivalent to a $\vec p$-positive formula in a finite
          neighborhood frame $\mathfrak{F}$ iff it is preserved under $\vec p$-directed
          bisimulations between models on $\mathfrak{F}$ (cf.~\cref{tab-pos-char})?
    \item Describe monotone formulas in non-monotone logics such as $\mathrm{E}$ and
          $\rm \mathrm{E} C$.
  \end{itemize}

  \bibliographystyle{alpha} \bibliography{defs/bib.bib}
\end{document}